\documentclass[11pt,oneside]{preprintCL}
\usepackage{comment}
\usepackage{hyperref}
\usepackage{breakurl}
\usepackage{times}
\usepackage{microtype}
\usepackage{booktabs}
\usepackage{tikz}
\usepackage{extarrows}
\usepackage{mathtools} 

\usepackage{amsmath}
\makeatletter
\let\over\@@over
\let\atop\@@atop
\makeatother
\usepackage{amssymb}
\usepackage{mathrsfs}
\usepackage{mhequ}
\usepackage{mhenvs}
\usepackage{mhsymb}

\usepackage{enumitem}

\definecolor{darkgreen}{rgb}{0.1,0.7,0.1}
\hypersetup{
citecolor=blue,
colorlinks=true, 
urlcolor= blue, 
linkcolor= black, 
bookmarksopen=true, 
pdftitle={Asymmetric bridges and KPZ}, 
}

\newtheorem{assumption}[lemma]{Assumption}

\DeclareMathOperator*{\esslim}{ess lim}

\newcommand{\argsh}{\mbox{argsh\,}}

\def\E{\mathbb{E}}
\def\P{\mathbb{P}}

\def\tun{\mathbf{1}}

\newcommand{\bbD}{\mathbb{D}}

\newcommand{\bbP}{\mathbb{P}}

\newcommand{\cC}{\mathcal{C}}

\newcommand{\cF}{\mathcal{F}}

\newcommand{\cI}{\mathcal{I}}
\newcommand{\cJ}{\mathcal{J}}

\newcommand{\cL}{\mathcal{L}}
\newcommand{\cM}{\mathcal{M}}
\newcommand{\cN}{\mathcal{N}}
\newcommand{\cO}{\mathcal{O}}

\newcommand{\ccM}{\mathscr{M}}

\begin{document}

\title{Weakly asymmetric bridges\\
	and the KPZ equation}
\author{Cyril Labb\'e}
\institute{Universit\'e Paris Dauphine\footnote{PSL Research University, CNRS, UMR 7534, CEREMADE, 75016 Paris, France.\\ Part of this work was carried out while the author was a Research Fellow at the University of Warwick.}}

\maketitle

\begin{abstract}
	We consider the corner growth dynamics on discrete bridges from $(0,0)$ to $(2N,0)$, or equivalently, the weakly asymmetric simple exclusion process with $N$ particles on $2N$ sites. We take an asymmetry of order $N^{-\alpha}$ with $\alpha \in (0,1)$ and provide a complete description of the asymptotic behaviour of this model. In particular, we show that the hydrodynamic limit of the density of particles is given by the inviscid Burgers equation with \textit{zero-flux} boundary condition. When the interface starts from the flat initial profile, we show that KPZ fluctuations occur whenever $\alpha \in (0,1/3]$. In the particular regime $\alpha =1/3$, these KPZ fluctuations suddenly vanish at a deterministic time.
\end{abstract}



\tableofcontents

\section{Introduction}

Consider the discrete set $\cC_N$ of bridges from $(0,0)$ to $(2N,0)$ which have slope $\pm 1$ on each interval $[k,k+1)$. We are interested in the scaling behaviour of the $\cC_N$-valued continuous-time Markov chain obtained by running the so-called \textit{corner growth dynamics}: at rate $p_N$ (resp.~$1-p_N$) every downwards corner (resp.~upwards corner) flips into its opposite, see Figure \ref{Fig1}. The generator of this chain can be written as follows:
\begin{equ}
\cL_N f(S) = \sum_{k=1}^{2N-1} \Big(p_N\tun_{\{\Delta S(k) = +2\}} + (1-p_N) \tun_{\{\Delta S(k) = -2\}}\Big) \big( f(S^k) - f(S)\big)\;,
\end{equ}
where $\Delta$ denotes the discrete Laplacian $\Delta S(k) = S(k+1) - 2 S(k) + S(k-1)$ and $S^k$ is the bridge obtained from $S$ by flipping the corner at site $k$ (if any):
\begin{equ}
S^k(\ell) = \begin{cases} S(\ell) & \ell \ne k\;,\\
S(k) + \Delta S(k)& \ell = k\;.
\end{cases}
\end{equ}

If we set $\eta(t,k)=(S(t,k)-S(t,k-1)+1)/2 \in \{0,1\}$, then $\eta$ is the asymmetric simple exclusion process with $N$ particles on $\{1,\ldots,2N\}$ and $S$ is its evolving height function. Notice that we impose a ``zero-flux" boundary condition at the level of the exclusion process: a particle at site $1$, resp.~$2N$, is not allowed to jump to its left, resp.~right. While many results have been established on this process (and its related versions on the infinite lattice $\Z$, on the torus, with reservoirs, etc.), the regime of weak asymmetry where
\begin{equ}\label{Eq:pN}
p_N \sim \frac12 + \frac1{(2N)^\alpha} + \cO\Big(\frac1{N^{2\alpha}}\Big)\;,\quad \alpha\in (0,1)\;,
\end{equ}
seems to have never been investigated before. This choice of asymmetry, combined with our boundary conditions, gives rise to scaling behaviours which, to the best of our knowledge, have not been observed previously in the literature, see below. A complete description of the full range $\alpha \in (0,\infty)$ of scaling limits of this model is presented in a companion paper~\cite{LabbeReview}. In the present paper, we will \textit{always} assume $\alpha \in (0,1)$.

\begin{figure}
\begin{minipage}{0.5\linewidth}
\centering\label{Fig1}
	\begin{tikzpicture}[scale=1.3]
	
	\draw[-,thin,color=gray] (0,0)node[below]{\tiny$0$} -> (3,0) node[below]{\tiny$2N$};
	
	\draw[-,gray] (0,0) -- (1.5,1.5) -- (3,0);
	\draw[-,gray] (0,0) -- (1.5,-1.5) -- (3,0);
	
	\draw[-,thick,color=black] (0,0) -- (0.25,0.25) -- (0.5,0) -- (0.75,-0.25) -- (1,-0.5) -- (1.25,-0.25) -- (1.5,0) -- (1.75,0.25) -- (2,0) -- (2.25,0.25) -- (2.5,0.5) -- (2.75,0.25) -- (3,0);
	
	\draw[-,style=dotted] (0.75,-0.25) -- (1,0)node[above] {\tiny \mbox{rate} $p_N$} -- (1.25,-0.25);
	\draw[->] (1,-0.4) -- (1,-0.1);
	
	\draw[-,style=dotted] (1.5,0) -- (1.75,-0.25)node[below] {\tiny \mbox{rate} $1-p_N$}  -- (2,0);
	\draw[->] (1.75,0.15) -- (1.75,-0.15);
	
	
	\end{tikzpicture}
	\caption{An example of interface.}
	
\end{minipage}
\begin{minipage}{0.5\linewidth}
	
	\begin{tikzpicture}[scale=1.3]
\draw[-,gray] (0,0) node[below]{\tiny$0$} -- (3,0) node[below]{\tiny$2N$};
\draw[-,gray] (0,0) -- (1.5,1.5) -- (3,0);
\draw[-,gray] (0,0) -- (1.5,-1.5) -- (3,0);

\draw[-,thick] (0,0) -- (0.05,0.05) -- (0.1,0.1) -- (0.15,0.15) -- (0.2,0.2) -- (0.25,0.25) -- (0.3,0.3) -- (0.35,0.35) -- (0.4,0.4) -- (0.45,0.45) -- (0.5,0.5) -- (0.55,0.55) -- (0.6,0.6) -- (0.65,0.65) -- (0.7,0.7) -- (0.75,0.75) -- (0.8,0.8) -- (0.85,0.85) -- (0.9,0.9) -- (0.95,0.95) -- (1,1) -- (1.05,1.05) -- (1.1,1.1) -- (1.15,1.15) -- (1.2,1.1) -- (1.25,1.15) -- (1.3,1.2) -- (1.35,1.25) -- (1.4,1.3) -- (1.45,1.35) -- (1.5,1.3) -- (1.55,1.35) -- (1.6,1.3) -- (1.65,1.25) -- (1.7,1.2) -- (1.75,1.25) -- (1.8,1.2) -- (1.85,1.15) -- (1.9,1.1) -- (1.95,1.05) -- (2,1) -- (2.05,0.95) -- (2.1,0.9) -- (2.15,0.85) -- (2.2,0.8) -- (2.25,0.75) -- (2.3,0.7) -- (2.35,0.65) -- (2.4,0.6) -- (2.45,0.55) -- (2.5,0.5) -- (2.55,0.45) -- (2.6, 0.4) -- (2.65,0.35) -- (2.7,0.3) -- (2.75,0.25) -- (2.8,0.2) -- (2.85,0.15) -- (2.9,0.1) -- (2.95,0.05) -- (3,0);

\draw[-,red] (0,0)--(1.2,1.2);
\draw[-,red][domain=1.2:1.8] plot(\x, {(\x-1.2)*(1.8-\x) + 1.2});
\draw[-,red] (1.8,1.2)--(3,0);

\draw[-,gray] (1.5,-0.025) -- (1.5,0.025);
\draw[-,gray] (1.5,0) node[below] {\tiny$N$};
\draw[<->,gray] (1.15,-0.4) -- (1.5,-0.4) node[below]{$\cO(N^{\alpha})$} -- (1.85,-0.4);

\draw[<->,red] (1.5,0.03) -- (1.5,0.3) node[right]{\tiny$N - \cO(N^\alpha)$} -- (1.5,1.29);

\draw[<->,black] (2.1,1.2) -- (2.1,1.28) node[right]{\tiny$\cO(N^{\frac{\alpha}{2}})$} -- (2.1,1.36);

\draw[-,dashed,gray] (1,1.05) -- (2,1.05) -- (2,1.6) -- (1,1.6) -- (1,1.05);

\end{tikzpicture}
\caption{Scaling limit under $\mu_N$ for $\alpha\in (0,1)$. The red curve stands for $\Sigma_\alpha^N$.}
\end{minipage}
\end{figure}

For convenience, we parametrise our asymmetry by setting
\begin{equ}
\frac{p_N}{1-p_N} = e^{\frac{4}{(2N)^\alpha}}\;,\qquad \alpha \in (0,1)\;.
\end{equ}
Plainly, this yields (\ref{Eq:pN}). Our dynamics admits a unique reversible probability measure:
\begin{equ}\label{Eq:muN}
\mu_N(S) = \frac1{Z_N} \Big(\frac{p_N}{1-p_N}\Big)^{\frac{1}{2}A(S)}\;,
\end{equ}
where $A(S) = \sum_{k=1}^{2N} S(k)$ is the area under the discrete bridge $S$, and $Z_N$ is a normalisation constant. Notice that the dynamics is reversible w.r.t.~$\mu_N$ even if the jump rates are asymmetric: this feature of the model is a consequence of the zero-flux boundary condition.

\subsection{Invariant measure and equilibrium fluctuations}

Let us start with a description of the scaling limit of the invariant measure $\mu_N$. As opposed to the regime $\alpha \geq 1$, in order to see a non-trivial asymptotic behaviour when $\alpha \in (0,1)$, one needs to zoom in a window of order $N^\alpha$ around the center of the lattice. Indeed, the interface is essentially stuck to the maximal path $x\mapsto x \wedge (N-x)$ except in this window. To write a precise statement, we need to introduce the curve $\Sigma^N_\alpha$ around which the fluctuations occur under $\mu_N$. Let $L(x)=\log\cosh x$. For all $k\in\{0,\ldots,2N\}$, we set $x_k=(k-N)/(2N)^\alpha$ as well as
\begin{equ}\label{Eq:DefSigmaN}
\Sigma^N_\alpha(x_k) = \sum_{i=1}^{k} L'(h^N_i)\;,\quad h^N_i = \frac{2}{(2N)^\alpha}\Big(N-i+\frac{1}{2}\Big)\;,\quad i\in\{1,\ldots,2N\}\;.
\end{equ}
In between these discrete values $x_k$'s, $\Sigma^N_\alpha$ is defined by linear interpolation. Then, we set
\begin{equ}
u^N(x) := \frac{S(N + x(2N)^\alpha)-\Sigma_\alpha^N(x)}{(2N)^{\frac{\alpha}{2}}} \;,\quad \forall x\in I^N_\alpha:=[-N/(2N)^\alpha,N/(2N)^\alpha]\;,
\end{equ}
and $u^N(x):=0$ for all $x\in \R\backslash I^N_\alpha$.

\begin{theorem}[\cite{LabbeReview}]\label{Th:Static}
Take $\alpha \in (0,1)$. The law of the process $u^N$ under $\mu_N$ converges to the law of the centred Gaussian process $(B_\alpha(x),x\in\R)$, with covariance
\begin{equ}
\E\big[B_\alpha(x)B_\alpha(y)\big] = \frac{q_\alpha(-\infty,x)\, q_\alpha(y,+\infty)}{q_\alpha(-\infty,+\infty)}\;,\quad\forall x\leq y \in \R\;,
\end{equ}
where $q_\alpha(x,y) = \int_{x \wedge y}^{x\vee y} L''(2u) du$.
\end{theorem}

\noindent The process $B_\alpha$ can be obtained by composing a Brownian bridge on $(0,1)$ with a deterministic time-change that maps $\R$ onto $(0,1)$.

\medskip

We turn our attention to the dynamics of this model. In the sequel, $\dot{W}$ will denote a space-time white noise on $\R$. If one starts the interface at equilibrium, then classical techniques based on the Boltzmann-Gibbs principle ensure that the scaling limit (after diffusive rescaling in height-space-time) is given by the solution of a stochastic heat equation. More precisely, if we set
$$u^N(t,x) := \frac{S(t(2N)^{2\alpha},N + x(2N)^\alpha)-\Sigma_\alpha^N(x)}{(2N)^{\frac{\alpha}{2}}} \;,\quad \forall x\in I^N_\alpha\;,\quad \forall t\ge 0\;,$$
then we have the following result.
\begin{theorem}[\cite{LabbeReview}]
The sequence $u^N$, starting from the equilibrium measure $\mu_N$, converges in distribution in the Skorohod space $\bbD([0,\infty),\cC(\R))$ towards the solution $u$ of the following stochastic PDE
\begin{equ}\label{SHE3}
\partial_t u = \frac{1}{2} \partial^2_x u -2 \partial_x \Sigma_\alpha\, \partial_x u +  \sqrt{1-(\partial_x \Sigma_\alpha)^2}\,\dot{W}\;,\quad x\in\R\;,
\end{equ}
started from an independent realisation of $B_\alpha$.
\end{theorem}

Here $\Sigma_\alpha(x) = \lim_{N\rightarrow\infty} (\Sigma^N_\alpha(x)-N)/(2N)^\alpha$ for all $x\in\R$.

\subsection{Hydrodynamic limit}

If one starts the interface out of equilibrium then one needs to derive the hydrodynamic limit, and this is one the main results of this paper. Let us rescale the interface in the following way
\begin{equ}
m^N(t,x) := \frac{S\big(t(2N)^{1+\alpha},x2N\big)}{2N}\;,\quad t\geq 0\;,\quad x\in[0,1]\;,
\end{equ}
and let us introduce the density of particles by setting
\begin{equ}\label{Eq:DefrhoN}
\rho^N(t,dx)=\frac1{2N}\sum_{k=1}^{2N} \eta(t(2N)^{1+\alpha},k)\ \delta_{\frac{k}{2N}}(dx)\;.
\end{equ}
Let $\cM$ be the set of measures on $[0,1]$ with total mass less or equal to $1$, equipped with the weak topology. At any time $t$, $\rho^N(t)$ has total mass $1/2$. Also, $m^N(t,\cdot)$ is $1$-Lipschitz at any time; therefore, any sequence of initial conditions $m^N(0,\cdot)$ is tight in $\cC([0,1],\R)$, all the limit points are $1$-Lipschitz and their derivatives are in $L^\infty$.

\begin{theorem}[Hydrodynamic limit]\label{Th:Hydro}
Take $\alpha \in (0,1)$. Assume that $m^N(0,\cdot)$ is deterministic and converges to some limiting profile $m(0,\cdot)$ for the supremum norm. Then, $\rho^N$ converges in law for the Skorohod topology on $\bbD([0,\infty),\cM)$ towards $\rho(t,dx)=\eta(t,x) dx$ where $\eta$ is the entropy solution of the inviscid Burgers equation with zero-flux boundary condition:
\begin{equ}\label{PDEBurgersDensity}
	\begin{cases}\partial_t \eta = 2\partial_x\big(\eta(1-\eta)\big) \;,\quad x\in(0,1)\;,\quad t>0\\
	\big(\eta(1-\eta)\big)(t,x) = 0\;,\quad x\in\{0,1\}\;,\quad t>0\;,\\
	\eta(0,\cdot) = \big(\partial_x m(0,\cdot)+1\big)/2\;.\end{cases}
\end{equ}
Furthermore, the sequence $m^N$ converges in $\bbD([0,\infty),\cC([0,1]))$ towards the integrated solution $m(t,x) = \int_{y=0}^x (2\eta(t,y) -1) dy$ of (\ref{PDEBurgersDensity}).
\end{theorem}

Let us mention that this hyperbolic PDE does not admit unique weak solutions so that one needs to consider entropy solutions. When the domain is bounded, the first solution theory was proposed by Bardos, Le Roux and N\'ed\'elec~\cite{Bardos} in the setting of Dirichlet boundary conditions: in that case, the solution does not necessarily fulfill the prescribed boundary conditions, but satisfies instead the so-called BLN conditions at the boundaries, see (\ref{Eq:BLN}). The solution theory with zero-flux boundary condition was introduced by B\"urger, Frid and Karlsen~\cite{BFK07}. To the best of our knowledge, this type of boundary conditions for the inviscid Burgers equation has not been considered before in scaling limits of particle systems.

To prove this theorem, we actually use a slightly indirect approach. First, we show that the solution of (\ref{PDEBurgersDensity}) coincides with the solution of the same equation with appropriate Dirichlet boundary conditions: namely, we impose $\eta(t,0)=1$ and $\eta(t,1)=0$. As mentioned above, these Dirichlet boundary conditions have to be interpreted in the BLN sense. Second, we prove convergence of our system towards this alternative PDE. This second part of the proof is in the spirit of the works of Rezakhanlou~\cite{Reza} and Bahadoran~\cite{Baha}. Let us outline the main differences. Here we deal with an asymmetry that vanishes with the size of the system, while the asymmetry is fixed in these two works. This induces our time-scaling $N^{1+\alpha}$ as opposed to the usual Euler scaling $N$, and some specific arguments are needed in the proof. Furthermore, our initial conditions are in general not given by product measures while this is the standing assumption in the two aforementioned works. However, using the $L^1$ contractivity of the solution map associated to (\ref{PDEBurgersDensity}) and a simple approximation argument, we obtain our convergence result.

Let us give a heuristic explanation for the Dirichlet boundary conditions chosen in the alternative PDE. Our particle system could essentially be obtained as the restriction to $\{1,\ldots,2N\}$ of the same dynamics on $\Z$, starting with the same configuration on $\{1,\ldots,2N\}$ and with particles on every negative site, and no particles on every site larger than $2N$. Indeed, with such an initial configuration and given the asymmetry of the dynamics, at any positive time there would still be essentially no particles after site $2N$ and no holes before site $1$. This corresponds to the chosen Dirichlet boundary conditions.

An important feature of this hyperbolic PDE is that, for any initial condition, it reaches the macroscopic stationary state $x\mapsto x\wedge(1-x)$ in finite time. Notice that for $\alpha \geq 1$, the hydrodynamic limit is given by a heat equation so that it takes infinite time to reach the stationarity state, see~\cite{LabbeReview}. This feature of the PDE has a non-trivial consequence at the level of the KPZ fluctuations, that we now investigate.

\subsection{KPZ fluctuations}

For convenience, we consider the flat initial condition:
\begin{equ}
S(0,k) = k \mbox{ mod } 2\;,\qquad \forall k\in\{0,\ldots,2N\}\;.
\end{equ}
In that case, the height function associated to the solution of (\ref{PDEBurgersDensity}) is given by:
\begin{equ}\label{Eq:ExplicitHydroFlat}
m(t,x) = x\wedge (1-x) \wedge t\;,\quad t > 0\;,\quad x\in [0,1]\;,
\end{equ}
see Figure \ref{FigKPZ}. The macroscopic stationary state is reached at the finite time $t_f=1/2$.
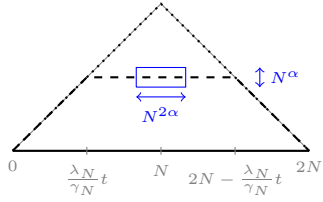
\begin{figure}
\begin{center}
\begin{tikzpicture}[scale=1.3]
\draw[-,gray] (0,0) node[below]{\tiny $0$} -- (3,0) node[below]{\tiny $2N$};
\draw[-,gray] (0,0) -- (1.5,1.5) -- (3,0);

\draw[-,thick] (0,0) -- (3,0);

\draw[-,dashed,thick] (0,0) -- (0.75,0.75) -- (2.25,0.75) -- (3,0);

\draw[-,dotted,thick] (0,0) -- (1.5,1.5) -- (3,0);

\draw[gray] (1.5,-0.04)node[below]{\tiny $N$}--(1.5,0.04);
\draw[gray] (0.75,-0.04)node[below]{\tiny$\frac{\lambda_N}{\gamma_N}t$}--(0.75,0.04);
\draw[gray] (2.25,-0.04)node[below]{\tiny$2N-\frac{\lambda_N}{\gamma_N}t$}--(2.25,0.04);

\draw[thin,blue] (1.25,0.65) -- (1.75,0.65) -- (1.75,0.85) -- (1.25, 0.85) -- (1.25,0.65);
\draw[<->,blue] (1.25,0.55) --(1.5,0.55)node[below]{\tiny $N^{2\alpha}$} -- (1.75,0.55);
\draw[<->,blue] (2.5,0.65) --(2.5,0.75)node[right]{\tiny $N^{\alpha}$} -- (2.5,0.85);

\end{tikzpicture}
\end{center}
\caption{A plot of (\ref{Eq:ExplicitHydroFlat}): the bold black line is the initial condition, the dashed line is the solution at some time $0 < t < 1/2$, and the dotted line is the solution at the terminal time $1/2$. The blue box corresponds to the window where we see KPZ fluctuations.}\label{FigKPZ}
\end{figure}

One would naturally expect KPZ fluctuations to occur around this growing front. Let us recall a famous result of Bertini and Giacomin in that direction. They consider the WASEP on the infinite lattice $\Z$, with upwards asymmetry $\epsilon$. Starting from a flat initial profile, the hydrodynamic limit grows uniformly in space and at constant speed. Bertini and Giacomin look at fluctuations around this hydrodynamic limit according to the following scaling: $\epsilon$ in height, $\epsilon^{-2}$ in space and $\epsilon^{-4}$ in time. They show that the rescaled interface converges to the Hopf-Cole solution of the KPZ equation on $\R$. We stress that this height-space-time scaling is rigid in order to observe KPZ fluctuations in the exclusion process.

In our setting, one would therefore expect a similar result to occur: the rescaling should be given by $1/(2N)^\alpha$ in height, $(2N)^{2\alpha}$ in space and $(2N)^{4\alpha}$ in time. Since the lattice is finite, one easily sees that $\alpha$ should not be taken larger than $1/2$ and therefore, it would be reasonable to make the following guess: for $\alpha \in (0,1/2)$, the fluctuations around the hydrodynamic limit are given by the KPZ equation on the line; for $\alpha =1/2$, they are given by the KPZ equation on $[0,1]$ with Dirichlet boundary conditions.

It turns out that KPZ fluctuations only occur for $\alpha \in (0,1/3]$, and that for the specific choice $\alpha = 1/3$ these fluctuations suddenly vanish at a deterministic finite time. This is the main result of the present paper. To state the precise result, we set
\begin{equ}\label{Eq:hN}
h^N(t,x) := \gamma_N S\big(t(2N)^{4\alpha},N + x(2N)^{2\alpha}\big) - \lambda_N t \;,
\end{equ}
where
\begin{equs}\label{Eq:GammaLambda}
	\gamma_N &:= \frac12 \log \frac{p_N}{1-p_N} \sim \frac{2}{(2N)^\alpha}\;,\\
	c_N &:= \frac{(2N)^{4\alpha}}{e^{\gamma_N} + e^{-\gamma_N}} \sim \frac{(2N)^{4\alpha}}{2} \;,\\
	\lambda_N &:= c_N( e^{\gamma_N} -2 + e^{-\gamma_N}) \sim 2(2N)^{2\alpha}\;.
\end{equs}

\begin{theorem}[KPZ fluctuations]\label{Th:KPZ}
Take $\alpha \in (0,1/3]$ and consider the flat initial condition. As $N\rightarrow\infty$, the sequence $h^N$ converges in distribution to the solution of the KPZ equation:
\begin{equ}\label{KPZ}
\begin{cases}
\partial_t h = \frac{1}{2} \partial^2_x h - (\partial_x h)^2 + \dot{W}\;,\quad x\in\R\;,\quad t > 0\;,\\
h(0,x)=0\;.
\end{cases}
\end{equ}
The convergence holds in $\bbD\big([0,T),\cC(\R)\big)$ where $T=1/2$ when $\alpha=1/3$, and $T=\infty$ when $\alpha < 1/3$. Here $\bbD\big([0,T),\cC(\R)\big)$ is endowed with the topology of uniform convergence on compact subsets of $[0,T)$.
\end{theorem}

Let us comment on this result. First of all, the restriction $\alpha \leq 1/3$ is a consequence of the finite-time convergence to equilibrium of the hydrodynamic limit. Indeed, the system takes a time of order $N^{1+\alpha}$ to reach the stationary state. This stationary state is reversible, and therefore, the irreversible KPZ fluctuations cannot be observed therein (this is quite different from the picture arising on the whole line $\Z$ where KPZ fluctuations occur under the invariant measure given by a product of Bernoulli measures). Since the KPZ fluctuations evolve on the time-scale $N^{4\alpha}$, one needs to restrict to $4\alpha \leq 1+\alpha$ in order for the fluctuations not to go faster than the hydrodynamic limit. This explains the restriction $\alpha \leq 1/3$.

For $\alpha = 1/3$, $T$ is the time needed by the hydrodynamic limit to reach the stationary state. Indeed, in that case the time-scale of the hydrodynamic limit coincides with the time-scale of the KPZ fluctuations. Although one could have thought that the fluctuations continuously vanish as $t\uparrow T$, our result show that they don't: the limiting fluctuations are given by the solution of the KPZ equation, restricted to the time interval $[0,T)$. This means that the fluctuations suddenly vanish at time $T$; let us give a simple explanation for this phenomenon. At any time $t\in [0,T)$, the particle system is split into three zones: a high density zone $\{1,\ldots,\frac{\lambda_N}{\gamma_N} t\}$, a low density zone $\{2N- \frac{\lambda_N}{\gamma_N} t, \ldots,2N\}$ and, in between, the bulk where the density of particles is approximately $1/2$, we refer to Figure \ref{FigKPZ}. The KPZ fluctuations occur in a window of order $N^{2\alpha}$ around the middle point of the bulk: from the point of view of this window, the boundaries of the bulk are ``at infinity" but move ``at infinite speed". Therefore, inside this window the system does not feel the effect of the boundary conditions until the very final time $T$ when the boundaries of the bulk merge.

\medskip

Let us recall that the KPZ equation is a singular SPDE: indeed, the solution of the linearised equation is not differentiable in space so that the non-linear term would involve the square of a distribution. While it was introduced in the physics literature~\cite{KPZ86} by Kardar, Parisi and Zhang, a first rigorous definition was given by Bertini and Giacomin~\cite{BG97} through the so-called Hopf-Cole transform $h\mapsto \xi = e^{-2 h}$ that maps formally the equation (\ref{KPZ}) onto
\begin{equ}\label{mSHE}
\begin{cases}
	\partial_t \xi = \frac{1}{2} \partial^2_x \xi +  2 \xi\dot{W}\;,\quad x\in\R\;,\quad t > 0\;,\\
	\xi(0,x)=1\;.
\end{cases}
\end{equ}
This SPDE is usually referred to as the multiplicative stochastic heat equation: it admits a notion of solution via It\^o integration. M\"uller~\cite{Mueller} showed that the solution is strictly positive at all times, if the initial condition is non-negative and non-zero. Then, one can define the solution of (\ref{KPZ}) to be $h:=-(\log \xi) / 2$. This is the notion of solution that we consider in Theorem \ref{Th:KPZ}.\\
There exists a more direct definition of this SPDE (restricted to a bounded domain) due to Hairer~\cite{HairerKPZ, HairerReg} via his theory of regularity structures. Let us also mention the notion of ``energy solution" introduced by Gon\c{c}alves and Jara~\cite{GJEnergy}, for which uniqueness has been proved by Gubinelli and Perkowski~\cite{GubPerEnergy}. It provides a new framework for characterising the solution to the KPZ equation but it requires the equation to be taken under its stationary measure.\\
For related convergence results towards KPZ, we refer to Amir, Corwin and Quastel~\cite{ACQ}, Dembo and Tsai~\cite{DemboTsai}, Corwin and Tsai~\cite{CT} and Corwin, Shen and Tsai~\cite{CST}. We also point out the reviews of Corwin~\cite{Corwin}, Quastel~\cite{QuastelKPZ} and Spohn~\cite{Spohn}.

\bigskip

The paper is organised as follows. In Section \ref{Section:CVEq}, we derive the hydrodynamic limit and in Section \ref{Section:KPZ} we prove the convergence of the fluctuations to the KPZ equation. Some technical proofs are postponed to the Appendix. Let us mention that we do not provide here the proofs of a few intermediate lemmas as they are classical in the literature: the one and two blocks estimates in Section \ref{Section:CVEq} and the bounds on the moments of space-time increments in Section \ref{Section:KPZ}. However, all the details on these lemmas can be found in the companion paper~\cite{LabbeReview}.

\paragraph{Acknowledgements.}
I am indebted to Reda Chhaibi for some deep discussions on this work at an early stage of the project. I would also like to thank Christophe Bahadoran for a fruitful discussion on the notion of entropy solutions for the inviscid Burgers equation, Nikos Zygouras for pointing out the article~\cite{DobrushinHryniv} and Julien Reygner for his helpful comments on a preliminary version of the paper. Finally, I am thankful to the anonymous referee for his/her comments that helped to improve the presentation of this article.

\section{Hydrodynamic limit}\label{Section:CVEq}

In this section, we work at the level of the sped up particle system $\eta^N_t,t\geq 0$ where
\begin{equ}
\eta^N_t(k) = \frac{1+X(t(2N)^{1+\alpha},k)}{2}\;,\qquad X(t,k) = S(t,k)-S(t,k-1)\;,
\end{equ}
for $k\in\{1,\ldots,2N\}$. Notice that our dynamics still makes sense when the total number of particles is any integer between $0$ and $2N$. We let $\tau_k$ denote the shift by $k\in\Z$, namely
\begin{equ}
\tau_k \eta := \big(\eta(k+1),\eta(k+2),\ldots,\eta(k-1),\eta(k)\big)\;,
\end{equ}
where indices are taken modulo $2N$. We let $T_\ell(i):=\{i-\ell,i-\ell+1,\ldots,i+\ell\}$ be the box of size $2\ell+1$ around site $i$, and for any sequence $a(k),k\in\Z$, we define its average over $T_\ell(i)$ as follows:
\begin{equ}
\ccM_{T_\ell(i)} a := \frac1{2\ell+1}\sum_{k=i-\ell}^{i+\ell} a(k)\;.
\end{equ}

Let us now present the notion of solution that we consider for the Burgers equation with zero-flux boundary condition. This material is taken from~\cite[Def.4]{BFK07}.

\begin{definition}\label{Def:EntropySolutionZeroFlux}
Let $\eta_0\in L^\infty(0,1)$. We say that $\eta\in L^\infty\big((0,\infty)\times (0,1)\big)$ is an entropy solution of (\ref{PDEBurgersDensity}) if:\begin{enumerate}
\item For all $c\in[0,1]$ and all $\varphi\in\cC^\infty_c\big((0,\infty)\times(0,1),\R_+\big)$, we have
\begin{equs}
{}\int_0^\infty \int_0^1 \Big(&\big|\eta(t,x)-c\big|\partial_t \varphi(t,x)- 2\sgn(\eta(t,x)-c)\\
&\times\big((\eta(t,x)(1-\eta(t,x)) - c(1-c)\big)\partial_x\varphi(t,x)\Big) dx\,dt \geq 0\;,
\end{equs}
\item We have $\esslim_{t\downarrow 0} \int_0^1 \big| \eta(t,x)-\eta_0(x)\big| dx = 0$,
\item We have $\eta(t,x)(1-\eta(t,x)) = 0$ for almost all $t>0$ and all $x\in\{0,1\}$.
\end{enumerate}
\end{definition}
Let us mention that the first condition is sufficient to ensure that $\eta$ has a trace at the boundaries so that the third condition is meaningful. B\"urger, Frid and Karlsen~\cite[Sect. 4 and 5]{BFK07} show existence and uniqueness of entropy solutions with zero-flux boundary condition.

Let us now introduce the Burgers equation with some appropriate Dirichlet boundary conditions:
\begin{equ}\label{PDEBurgersDirichlet}
	\begin{cases}\partial_t \eta = 2\partial_x\big(\eta(1-\eta)\big) \;,\\
	\eta(t,0) = 1\;,\quad \eta(t,1)=0\;,\\
	\eta(0,\cdot) = \eta_0(\cdot)\;.\end{cases}
\end{equ}
The precise definition of the entropy solution of (\ref{PDEBurgersDirichlet}) is the following.
\begin{definition}\label{Def:EntropySolutionDirichlet}
Let $\eta_0\in L^\infty(0,1)$. We say that $\eta\in L^\infty\big((0,\infty)\times (0,1)\big)$ is an entropy solution of (\ref{PDEBurgersDirichlet}) if it satisfies conditions 1. and 2. from Definition \ref{Def:EntropySolutionZeroFlux} together with the so-called BLN conditions
\begin{equs}[Eq:BLN]
\sgn(\eta(t,0)-1)\big(\eta(t,0)(1-\eta(t,0)) - c(1-c)\big) &\geq 0\;,\quad \forall c\in [\eta(t,0),1]\;,\\
\sgn(\eta(t,1)-0)\big(\eta(t,1)(1-\eta(t,1)) - c(1-c)\big) &\leq 0\;,\quad \forall c\in [0,\eta(t,1)]\;,
\end{equs}
for almost all $t>0$.
\end{definition}

Here again, there is existence and uniqueness of entropy solutions of (\ref{PDEBurgersDirichlet}), see for instance~\cite[Sect. 2.7 and 2.8]{Ruzicka}.

\begin{proposition}\label{Prop:PDE}
The entropy solutions of (\ref{PDEBurgersDensity}) and (\ref{PDEBurgersDirichlet}) coincide.
\end{proposition}
\begin{proof}
Both solutions exist and are unique. Let us show that the solution of (\ref{PDEBurgersDirichlet}) satisfies the conditions of Definition \ref{Def:EntropySolutionZeroFlux}: actually, the two first conditions are automatically satisfied, so we focus on the third one. The BLN conditions above immediately imply that $\eta(t,0)$ and $\eta(t,1)$ are necessarily in $\{0,1\}$ for almost all $t>0$, so that the third condition is satisfied.
\end{proof}

As a consequence, we can choose the formulation (\ref{PDEBurgersDirichlet}) in the proof of our convergence result. Let us finally collect some properties of the solutions that we will use later on.

\begin{proposition}\label{Prop:EntropySolution}
Let $\eta_0\in L^\infty(0,1)$. A function $\eta\in L^\infty\big((0,\infty)\times (0,1)\big)$ is the entropy solution of (\ref{PDEBurgersDirichlet}) if and only if for all $c\in[0,1]$ and all $\varphi\in\cC^\infty_c\big([0,\infty)\times[0,1],\R_+\big)$ we have
\begin{equs}[Eq:EntropyCond]
{}&\int_0^\infty \int_0^1 \big((\eta(t,x)-c)^{\pm}\partial_t \varphi(t,x) + h^{\pm}(\eta(t,x),c) \partial_x\varphi(t,x)\big) dx\,dt\\
&+ \int_0^1 (\eta_0(x)-c)^\pm \varphi(0,x) dx+ 2 \int_0^\infty \Big((1-c)^\pm \varphi(t,0) + (0-c)^\pm \varphi(t,1)\Big) dt \geq 0\;,
\end{equs}
where $(x)^\pm$ denotes the positive/negative part of $x\in\R$, $\sgn^\pm(x)=\pm\tun_{(0,\infty)}(\pm x)$ and $h^\pm(\eta,c):=-2 \sgn^\pm(\eta-c) \big(\eta(1-\eta) - c(1-c)\big)$.\\
Furthermore for any $t>0$, the map $\eta_0\mapsto \eta(t)$ is $1$-Lipschitz in $L^1(0,1)$.
\end{proposition}
\begin{proof}
The notion of solution defined by (\ref{Eq:EntropyCond}) is introduced in~\cite[Def. 1]{Vovelle} and it is shown therein that it coincides with another notion of solution, originally due to Otto, which is based on boundary entropy-entropy flux pairs. It is then shown in~\cite[Th 7.31]{Ruzicka} that the latter notion of solution is equivalent with the notion of solution of Definition \ref{Def:EntropySolutionDirichlet}. This completes the proof of the first part of the statement. The Lipschitz continuity in $L^1(0,1)$ is proved in~\cite[Th. 3]{BFK07} for the Burgers equation with zero-flux boundary conditions. By Proposition \ref{Prop:PDE}, it also holds for (\ref{PDEBurgersDirichlet}), thus concluding the proof.
\end{proof}

\subsection{Proof of the convergence}\label{SectionHyperbo}

We let $\cM$ be the space of measures on $[0,1]$ with total mass at most $1$, endowed with the topology of weak convergence. Recall the process $\rho^N$ defined in (\ref{Eq:DefrhoN}).

\begin{proposition}\label{Prop:TightnessHyperbo}
Let $\iota_N$ be any probability measure on $\{0,1\}^{2N}$. The sequence of processes $(\rho^N_t, t\geq 0)$, starting from $\iota_N$, is tight in the space $\bbD([0,\infty),\cM)$. Furthermore, the associated sequence of processes $(m^N(t,x),t\geq 0,x\in[0,1])$ is tight in $\bbD([0,\infty),\cC([0,1]))$.
\end{proposition}

\noindent Note that for a generic measure $\iota_N$ on $\{0,1\}^{2N}$, $m^N(t,1)$ is not necessarily equal to $0$.

\begin{proof}
Let $\varphi\in\cC^2([0,1])$. It suffices to show that $\langle \rho^N_0, \varphi\rangle$ is tight in $\R$, and that for all $T>0$
\begin{equ}\label{Eq:TightnessHydro}
\varlimsup_{h\downarrow 0} \varlimsup_{N\rightarrow\infty} \E^N_{\iota_N}\Big[ \sup_{s,t\leq T, |t-s|\leq h}| \langle \rho^N_t-\rho^N_s, \varphi\rangle |\Big]=0\;.
\end{equ}
The former is immediate since $|\langle \rho^N_0, \varphi\rangle| \leq \|\varphi\|_\infty$. Regarding the latter, we let $L^N$ be the generator of our sped-up process and we write
\begin{equs}
\langle \rho^N_t-\rho^N_s, \varphi\rangle = \frac1{2N} \int_s^t \sum_{k=1}^{2N} \varphi(k) L^N \eta_r^N(k) dr + M^N_{s,t}(\varphi)\;,
\end{equs}
where $M^N_{s,t}(\varphi)$ is a martingale. Its bracket can be bounded almost surely as follows
\begin{equ}
\langle M^N_{s,\cdot}(\varphi) \rangle_t \leq \int_s^t \frac1{(2N)^2} \sum_{k=1}^{2N-1}\big(\nabla\varphi(k)\big)^2 (2N)^{1+\alpha} dr \lesssim \frac{t-s}{(2N)^{2-\alpha}}\;.
\end{equ}
Since the jumps of this martingale are bounded by a term of order $\|\varphi'\|_\infty / (2N)^2$, and since
$$\E^N_{\iota_N}\Big[ \sup_{t\in [s,s+h]\cap [0,T]}|M^N_{s,t}(\varphi)|\Big]\leq \E^N_{\iota_N}\Big[ \sup_{t\in [s,s+h]\cap [0,T]}|M^N_{s,t}(\varphi)|^2\Big]^\frac12\;,$$
the BDG inequality (\ref{Eq:BDG3}) ensures that we have
\begin{equ}\label{Eq:BdSupMgale}
\varlimsup_{N\rightarrow\infty} \sup_{s\in [0,T]} \E^N_{\iota_N}\Big[ \sup_{t\in [s,s+h]\cap [0,T]}|M^N_{s,t}(\varphi)|\Big]= 0\;.
\end{equ}
This being given, we observe that $M^N_{s,t}(\varphi)=M^N_{r,t}(\varphi) - M^N_{r,s}(\varphi)$ where $r$ is taken to be the largest element in $\{0,h,2h,\ldots, \lfloor\frac{T}{h}\rfloor h\}$ which is below $s$. Therefore, we have
\begin{equs}
\E^N_{\iota_N}\Big[ \sup_{0\le s\le t\le T, |t-s|\leq h}|M^N_{s,t}(\varphi)|\Big] &\le 2\, \E^N_{\iota_N}\Big[\sum_{r=0,h,\ldots,\lfloor\frac{T}{h}\rfloor h} \sup_{t\in [r,r+2h]\cap [0,T]} |M^N_{r,t}(\varphi)|\Big]\\
&\le 2 \sum_{r=0,h,\ldots,\lfloor\frac{T}{h}\rfloor h} \E^N_{\iota_N}\Big[ \sup_{t\in [r,r+2h]\cap [0,T]} |M^N_{r,t}(\varphi)|\Big]\\
&\le \frac{C}{h} \sup_{r\in[0,T]} \E^N_{\iota_N}\Big[ \sup_{t\in [r,r+2h]\cap [0,T]} |M^N_{r,t}(\varphi)|\Big]\;,
\end{equs}
for some $C>0$. Combining this with \eqref{Eq:BdSupMgale} we deduce that
\begin{equ}\label{Eq:TightnessHydro1}
\varlimsup_{h\downarrow 0} \varlimsup_{N\rightarrow\infty} \E^N_{\iota_N}\Big[ \sup_{0\le s\le t\le T, |t-s|\leq h}|M^N_{s,t}(\varphi)|\Big]=0\;.
\end{equ}
Let us bound the term involving the generator. Decomposing the jump rates into the symmetric part (of intensity $1-p_N$) and the totally asymmetric part (of intensity $2p_N-1$), we find
\begin{equs}
\frac1{2N} \sum_{k=1}^{2N} \varphi(k) L^N \eta(k)&=
-(2N)^\alpha (1-p_N) \sum_{k=1}^{2N-1} \nabla\eta(k) \nabla\varphi(k)\\
&\quad- (2N)^\alpha(2p_N-1) \sum_{k=1}^{2N-1} \eta(k+1)(1-\eta(k)) \nabla \varphi(k)\;.
\end{equs}
A simple integration by parts shows that the first term on the right is bounded by a term of order $N^{\alpha-1}$ while the second term is of order $1$. Consequently
\begin{equ}\label{Eq:TightnessHydro2}
\E^N_{\iota_N}\Big[\sup_{s,t \leq T, |t-s|\leq h} \Big|\frac1{2N} \int_s^t \sum_{k=1}^{2N} \varphi(k) L^N \eta_r(k) dr \Big|\Big] \lesssim h\;,
\end{equ}
uniformly over all $N\geq 1$ and all $h>0$. The l.h.s.~vanishes as $N\rightarrow\infty$ and $h\downarrow 0$. Combining (\ref{Eq:TightnessHydro1}) and (\ref{Eq:TightnessHydro2}), (\ref{Eq:TightnessHydro}) follows.\\
We turn to the tightness of the interface $m^N$. First, the profile $m^N(t,\cdot)$ is $1$-Lipschitz for all $t\geq 0$ and all $N\geq 1$. Second, we claim that for some $\beta\in(\alpha,1)$ 
\begin{equ}\label{Eq:IncrHyperbo}
\E^N_{\iota_N} \Big[ |m^N(t,k)-m^N(s,k)|^p \Big]^{\frac1{p}} \lesssim |t-s| + \frac1{N^{1-\beta}}\;,
\end{equ}
uniformly over all $0 \leq s \leq t \leq T$, all $k\in\{1,\ldots,2N\}$ and all $N\geq 1$. This being given, the arguments for proving tightness are classical: one introduces a piecewise linear time-interpolation $\bar{m}^N$ of $m^N$ and shows tightness for this process, and then one shows that the difference between $\bar{m}^N$ and $m^N$ is uniformly small. We are left with the proof of (\ref{Eq:IncrHyperbo}). Let $\psi:\R\rightarrow\R_+$ be a non-increasing, smooth function such that $\psi(x)=1$ for all $x\leq 0$ and $\psi(x)=0$ for all $x\geq 1$. Fix $\beta \in (\alpha,1)$. For any given $k\in\{1,\ldots,2N\}$, we define $\varphi^N_k:\{0,\ldots,2N\}\rightarrow\R$ by setting $\varphi^N_k(\ell) = \psi\big((\ell-k)/(2N)^\beta\big)$. Then, we observe that
\begin{equ}
\frac1{2N}\sum_{\ell=1}^{2N} \big(2\eta_t(\ell)-1\big) \varphi^N_k(\ell) = m^N(t,k) + \cO(N^{\beta-1})\;,
\end{equ}
uniformly over all $k\in\{1,\ldots,2N\}$ and all $t\geq 0$. Then, similar computations to those made in the first part of the proof show that
\begin{equ}
\E^N_{\iota_N}\Big[\Big|\frac1{2N}\sum_{\ell=1}^{2N}\big(\eta_t(\ell)-\eta_s(\ell)\big) \varphi^N_k(\ell)\Big|^p\Big]^{\frac1{p}} \lesssim (t-s) + \sqrt{\frac{t-s}{N^{1+\beta-\alpha}}} + \frac1{N^{1+\beta}}\;,
\end{equ}
uniformly over all $k$, all $0\leq s \leq t \leq T$ and all $N\geq 1$. This yields (\ref{Eq:IncrHyperbo}).
\end{proof}

The main step in the proof of Theorem \ref{Th:Hydro} is to prove the convergence of the density of particles, starting from a product measure satisfying the following assumption:
\begin{assumption}\label{Assumption:IC}
For all $N\geq 1$, the initial condition $\iota_N$ is a product measure on $\{0,1\}^{2N}$ of the form $\otimes_{k=1}^{2N} \mbox{Be}(f(k/2N))$, where $f:[0,1]\rightarrow[0,1]$ is assumed to be piecewise constant and does not depend on $N$.
\end{assumption}
Under this assumption and if the process starts from $\iota_N$, then $\rho^N_0$ converges to the deterministic limit $\rho_0(dx)=f(x)dx$.

\begin{theorem}\label{Th:HydroProd}
Under Assumption \ref{Assumption:IC}, the process $\rho^N$ converges in distribution in the Skorohod space $\bbD\big([0,\infty),\cM\big)$ to the deterministic process $(\eta(t,x)dx,t\geq 0)$, where $\eta$ is the entropy solution of (\ref{PDEBurgersDirichlet}) starting from $\eta_0=f$.
\end{theorem}

\noindent Given this result, the proof of the hydrodynamic limit is simple.
\begin{proof}[of Theorem \ref{Th:Hydro}]
Let $\iota_N$ be as in Theorem \ref{Th:HydroProd}. We know that $\rho^N$ converges to some limit $\rho$ and that $m^N$ is tight. Let $m$ be some limit point and let $m^{N_i}$ be an associated converging subsequence. By Skorohod's representation theorem, we can assume that $(\rho^{N_i},m^{N_i})$ converges almost surely to $(\rho,m)$. Recall that $\rho$ is of the form $\rho(t,x)=\eta(t,x)dx$. Our first goal is to show that $m(t,x) = \int_0^x (2\eta(t,y) - 1)dy$ for all $t,x$.

Fix $x_0\in (0,1)$. We introduce an approximation of the indicator of $(-\infty,x_0]$ by setting $\varphi_p(\cdot) = 1- \int_{-\infty}^\cdot P_{1/p}(y-x_0) dy$, $p\geq 1$ where $P_t$ is the heat kernel on $\R$ at time $t$. For each $p\geq 1$, $\varphi_p$ is smooth and for any $\delta > 0$ we have
\begin{equ}\label{Eq:phip}
\big\| \varphi_p - \tun_{[0,x_0]} \big\|_{L^1(0,1)} \rightarrow 0\;,\quad \sup_{f\in\cC^\delta([0,1])} \frac{\big| \langle f, \delta_{x_0} + \partial_x \varphi_p \rangle \big|}{\| f \|_{\cC^\delta}} \rightarrow 0\;,
\end{equ}
as $p\rightarrow\infty$. If we set $I(t,x_0)= m(t,x_0) - \int_0^{x_0} (2\eta(t,y) - 1)dy$ for some $x_0\in (0,1)$ and some $t>0$, then $| I(t,x_0) |$ is bounded by
\begin{equs}
{}&\big\| m(t)- m^{N_i}(t)\big\|_\infty+\big| \langle m^{N_i}(t),\delta_{x_0} + \partial_x \varphi_p \rangle \big|+ \big|\langle m^{N_i}(t),\partial_x\varphi_p\rangle+\langle2\rho^{N_i}_t-1,\varphi_p\rangle\big|\\
&+ 2\big| \langle \rho^{N_i}_t-\rho_t,\varphi_p \rangle \big|+ \big| \langle 2\rho(t)-1,\varphi_p - \tun_{[0,x_0]}\rangle\big|\;.
\end{equs}
Recall that $m^N$ is $1$-Lipschitz in space, so that the second term vanishes as $p\rightarrow\infty$ by (\ref{Eq:phip}). A discrete integration by parts shows that the third term vanishes as $N_i$ goes to $\infty$. The first and fourth terms vanish as $N_i\rightarrow\infty$ by the convergence of $m^{N_i}$ and $\rho^{N_i}$, and the last term is dealt with using (\ref{Eq:phip}). Choosing $p$ and then $N_i$ large enough, we deduce that $|I(t,x_0)|$ is almost surely as small as desired. This identifies completely the limit $m$ of any converging subsequence, under $\P^N_{\iota_N}$.

We are left with the extension of this convergence result to an arbitrary initial condition. Assume that $m^N(0,\cdot)$ converges to some profile $m_0(\cdot)$ for the supremum norm. Since each $m^N(0,\cdot)$ is $1$-Lipschitz, so is $m_0$. For all $\epsilon > 0$, one can find two profiles $m_0^{\epsilon,+}$ and $m_0^{\epsilon,-}$ which are $1$-Lipschitz, piecewise affine, start from $0$ and are such that:
\begin{itemize}
\item $m_0^{\epsilon,-}$ stays below $m_0$: $m_0-\epsilon \leq m_0^{\epsilon,-} \leq (m_0-\frac{\epsilon}{4})\vee(-x)\vee(x-1)\;,$
\item $m_0^{\epsilon,+}$ stays above $m_0$: $(m_0+\frac{\epsilon}{4})\wedge x \wedge(1-x) \leq m_0^{\epsilon,+} \leq m_0 + \epsilon\;,$
\item $\| \rho_0^{\epsilon,\pm} - \rho_0\|_{L^1} \rightarrow 0$ as $\epsilon \downarrow 0$, where $\rho_0^{\epsilon,\pm}(\cdot) = \big(\partial_x m_0^{\epsilon,\pm}(\cdot) +1\big)/2$.
\end{itemize}
Let us briefly explain how one can construct $m_0^{\epsilon,-}$, the construction of $m_0^{\epsilon,+}$ being similar. For simplicity, we let $V(x) := (-x)\vee(x-1)$. Let $n\ge 1$ be given. We subdivide $[0,1]$ into three sets:
$$I:=\{x: m_0(x) > V(x) + \epsilon/2\}\;,\quad J_1:= [0,1/2]\backslash I\;,\quad J_2 := (1/2,1] \backslash I\;.$$
On $J_1$ and $J_2$, we set $m_0^{\epsilon,-}(x) = V(x)$. On $I\cap \{k/n: k=0,1,\ldots,n\}$, we set $m_0^{\epsilon,-}(x) =m_0(x) - \frac{\epsilon}{2}$. Then, we extend $m_0^{\epsilon,-}$ to the rest of $I$ by affine interpolation. The fact that $m_0$ is $1$-Lipschitz ensures that $m_0^{\epsilon,-}$ is also $1$-Lipschitz. If $n$ is large enough compared to $1/\epsilon$ we get the inequalities $m_0-\epsilon \leq m_0^{\epsilon,-} \leq (m_0-\frac{\epsilon}{4})\vee(-x)\vee(x-1)$. Regarding the convergence in $L^1$ of the density, we observe that $ \rho^{-,\epsilon}_0(x) = 0$ on $J_1$, $\rho^{-,\epsilon}_0(x) = 1$ on $J_2$ and
$$  \rho^{-,\epsilon}_0(x) = \frac1{|I(x)|} \int_{I(x)} \rho_0(u) du\;,\quad x\in I\;,$$
where $I(x) = I \cap [k/n,(k+1)/n)$ and $k$ is the integer part of $nx$. From there, we deduce that $\| \rho_0^{\epsilon,\pm} - \rho_0\|_{L^1(J_1)}=\int_{J_1} \rho_0(x)dx$. At this point, we decompose $J_1$ into its connected components $J_{1,\ell}$, $\ell=1,\ldots,L$ ranked in the increasing order. Notice that $J_{1,1}$ is an interval that starts at $0$ and $J_{1,L}$ potentially ends at $1/2$ (if $m_0(1/2) \le V(1/2)-\epsilon/2$). On the other hand, at both end points of $J_{1,\ell}$ for $\ell\notin \{1,L\}$, $m_0$ coincides with $V(x)-\epsilon/2$. As a consequence, $\int_{J_{1,\ell}} \rho_0(x)dx = 0$ for all $\ell \notin \{1,L\}$ while $\int_{J_{1,\ell}} \rho_0(x)dx \le \epsilon/4$ for $\ell \in \{1,L\}$. Henceforth $\| \rho_0^{\epsilon,\pm} - \rho_0\|_{L^1(J_1)} \le \epsilon/2$. Similarly, $\| \rho_0^{\epsilon,\pm} - \rho_0\|_{L^1(J_2)}=\int_{J_2} (1-\rho_0(x))dx$ which is also smaller than $\epsilon/2$. Finally, $\| \rho_0^{\epsilon,\pm} - \rho_0\|_{L^1(I)}$ goes to $0$ as $n\rightarrow\infty$: indeed, the almost everywhere differentiability of $x\mapsto \int_0^x \rho_0(u) du$ ensures that $\rho_0^{\epsilon,-}(x)$ goes to $\rho_0(x)$ for almost all $x\in I$, so that the dominated convergence theorem yields the asserted convergence.\\

Now consider a coupling $(m^{N,\epsilon,-},m^N,m^{N,\epsilon,+})$ of three instances of our height process which preserves the order of the interfaces and is such that $m^{N,\epsilon,\pm}(0,\cdot)$ is the height function associated with the particle density distributed as $\otimes_{k=1}^{2N} \mbox{Be}\big(\rho^{\epsilon,\pm}(k/2N)\big)$. We draw independently these two sets of Bernoulli r.v. It is simple to check that the probability of the event $m^{N,\epsilon,-}(0,\cdot) \leq m^N(0,\cdot) \leq m^{N,\epsilon,+}(0,\cdot)$ goes to $1$ as $N\rightarrow\infty$. By the order preserving property of the coupling, if these inequalities are satisfied at time $0$ they remain true at all times. Our convergence result applies to $m^{N,\epsilon,\pm}$ and, consequently, any limit point of the tight sequence $m^N$ is squeezed in $[m^{\epsilon,-},m^{\epsilon,+}]$ where $m^{\epsilon,\pm}$ is the integrated entropy solution of (\ref{PDEBurgersDirichlet}) starting from $m_0^{\epsilon,\pm}$. By the second part of Proposition \ref{Prop:PDE}, we deduce that $m^{\epsilon,\pm}$ converge, as $\epsilon\downarrow 0$, to the integrated entropy solution of (\ref{PDEBurgersDirichlet}) starting from $m_0$, thus concluding the proof. 
\end{proof}

To prove Theorem \ref{Th:HydroProd}, we need to show that the limit of any converging subsequence of $\rho^N$ is of the form $\rho(t,dx)=\eta(t,x)dx$ and that $\eta$ satisfies the entropy inequalities of Proposition \ref{Prop:EntropySolution}. To make appear the constant $c$ in these inequalities, the usual trick is to define a coupling of the particle system $\eta^N$ with another particle system $\zeta^N$ which is stationary with density $c$ so that, at large scales, one can replace the averages of $\zeta^N$ by $c$. Such a coupling has been defined by Rezakhanlou~\cite{Reza} in the case of the infinite lattice $\Z$. The specificity of the present setting comes from the boundary conditions of our system: one needs to choose carefully the flux of particles at $1$ and $2N$ for $\zeta^N$.

The precise definition of our coupling goes as follows. We set
\begin{equ}
p(1)=1-p_N\;,\quad p(-1)=p_N\;,\quad \mbox{and}\quad p(k)=0 \quad \forall k\ne \{-1,1\}\;,
\end{equ}
as well as $b(a,a')=a(1-a')$. We denote by $\eta^{k,\ell}$ the particle configuration obtained from $\eta$ by permuting the values $\eta(k)$ and $\eta(\ell)$. We also denote by $\eta \pm \delta_k$ the particle configuration which coincides with $\eta$ everywhere except at site $k$ where the occupation is set to $\eta(k) \pm 1$. Then, we define
\begin{equs}
\tilde{\cL}^{\mbox{\tiny bulk}}f(\eta,\zeta) &= (2N)^{1+\alpha}\sum_{k,\ell =1}^{2N} p(\ell-k)\times\\
&\Big[\big( b(\eta(k),\eta(\ell))\wedge b(\zeta(k),\zeta(\ell)) \big)\big( f(\eta^{k,\ell},\zeta^{k,\ell}) - f(\eta,\zeta)\big)\\
& + \big(b(\eta(k),\eta(\ell))- b(\eta(k),\eta(\ell))\wedge b(\zeta(k),\zeta(\ell)) \big)\big( f(\eta^{k,\ell},\zeta) - f(\eta,\zeta)\big)\\
& + \big( b(\zeta(k),\zeta(\ell))- b(\eta(k),\eta(\ell))\wedge b(\zeta(k),\zeta(\ell)) \big)\big( f(\eta,\zeta^{k,\ell}) - f(\eta,\zeta)\big) \Big]\;,
\end{equs}
and
\begin{equs}
\tilde{\cL}^{\mbox{\tiny bdry}} f(\eta,\zeta) &= (2N)^{1+\alpha}(2p_N-1) (1-c) \zeta(1) \big( f(\eta,\zeta-\delta_1) - f(\eta,\zeta)\big)\\
& + (2N)^{1+\alpha}(2p_N-1) c (1-\zeta(2N)) \big( f(\eta,\zeta+\delta_{2N}) - f(\eta,\zeta)\big)\;.
\end{equs}
We consider the stochastic process $(\eta^N_t,\zeta^N_t), t\geq 0$ associated to the generator $\tilde{\cL}=\tilde{\cL}^{\mbox{\tiny bulk}}+\tilde{\cL}^{\mbox{\tiny bdry}}$. From now on, we will always assume that $\eta^N_0$ has law $\iota_N$, where $\iota_N$ satisfies Assumption \ref{Assumption:IC}, and that $\zeta^N_0$ is distributed as a product of Bernoulli measures with parameter $c$. Furthermore, we will always assume that the coupling at time $0$ is such that 
\begin{equ}
\sgn(\eta^N_0(k)-\zeta^N_0(k)) = \sgn(f(k/2N)-c)\;,\quad \forall k\in\{1,\ldots,2N\}\;,
\end{equ}
where $f$ is the macroscopic density profile of Assumption \ref{Assumption:IC}. Such a coupling can be constructed by considering i.i.d.~r.v.~$U_1,\ldots,U_{2N}$ uniformly distributed over $[0,1]$, and by setting $\eta^N_0(k)$ (resp.~$\zeta^N_0(k)$) to $1$ if $U_k \le f(k/2N)$ (resp.~$U_k \le c$). We let $\tilde{\P}^N_{\iota_N,c}$ be the law of the process $(\eta^N,\zeta^N)$.
\begin{remark}
The process $\eta^N$ follows the dynamics of the WASEP with zero-flux boundary conditions. The process $\zeta^N$ follows the dynamics of the WASEP with some open boundary conditions chosen in a such a way that the process is stationary with density $c$. Actually, we prescribe the minimal jump rates at the boundary for the process to be stationary with density $c$: there is neither entering flux at $1$ nor exiting flux at $2N$. This choice is convenient for establishing the entropy inequalities. Let us also mention that the coupling is such that the order of $\zeta^N$ and $\eta^N$ is preserved. More precisely, if in both particle systems there is a particle which can attempt a jump from $k$ to $\ell$, then the jump times are simultaneous.
\end{remark}
It will actually be important to track the sign changes in the pair $(\eta^N,\zeta^N)$. To that end, we let $F_{k,\ell}(\eta,\zeta) = 1$ if $\eta(k)\geq \zeta(k)$ and $\eta(\ell)\geq \zeta(\ell)$; and $F_{k,\ell}(\eta,\zeta) = 0$ otherwise. We say that a subset $C$ of consecutive integers in $\{1,\ldots,2N\}$ is a cluster with constant sign if for all $k,\ell \in C$ we have $F_{k,\ell}(\eta,\zeta) =1$, or for all $k,\ell \in C$ we have $F_{k,\ell}(\zeta,\eta) = 1$. For a given configuration $(\eta,\zeta)$, we let $n$ be the minimal number of clusters needed to cover $\{1,\ldots,2N\}$: we will call $n$ the number of sign changes. There is not necessarily a unique choice of covering into $n$ clusters. Let $C(i),i\leq n$ be any such covering and let $1=k_1 <  k_2 < \ldots k_n < k_{n+1} = 2N+1$ be the integers such that $C(i)=\{k_i,k_{i+1}-1\}$.

\begin{lemma}\label{Lemma:Coupling}
Under $\tilde{\P}^N_{\iota_N,c}$, the process $\eta^N$ has law $\P^N_{\iota_N}$ while the process $\zeta^N$ is stationary with law $\otimes_{k=1}^{2N} \mbox{Be}(c)$. Furthermore, the number of sign changes $n(t)$ is smaller than $n(0)+3$ at all time $t\geq 0$.
\end{lemma}
\begin{proof}
It is simple to check the assertion on the laws of the marginals $\eta^N$ and $\zeta^N$. Regarding the number of sign changes, the key observation is the following. In the bulk $\{2,\ldots,2N-1\}$, to create a new sign change we need to have two consecutive sites $k,\ell$ such that $\eta^N(k)=\zeta^N(k)=1$, $\eta^N(\ell)=\zeta^N(\ell)=0$ and we need to let one particle jump from $k$ to $\ell$, but not both. However, our coupling does never allow such a jump. Therefore, the number of sign changes can only increase at the boundaries due to the interaction of $\zeta^N$ with the reservoirs: this can create at most $2$ new sign changes, thus concluding the proof.
\end{proof}

Assumption \ref{Assumption:IC} ensures the existence of a constant $C>0$ such that $n(0) < C$ almost surely for all $N\geq 1$. We now derive the entropy inequalities at the microscopic level. Recall that $\tau_k$ stands for the shift operator with periodic boundary conditions, and let $\langle u,v\rangle_N = (2N)^{-1} \sum_{k=1}^{2N} u(k/2N)v(k/2N)$ denote the discrete $L^2$ product.
\begin{lemma}[Microscopic inequalities]\label{Lemma:MicroIneq}
Let $\iota_N$ be a measure on $\{0,1\}^{2N}$ satisfying Assumption \ref{Assumption:IC}. For all $\varphi\in \cC^\infty_c([0,\infty)\times[0,1],\R_+)$, all $\delta > 0$ and all $c\in [0,1]$, we have $\lim_{N\rightarrow\infty} \tilde{\P}^N_{\iota_N,c}(\cI^N(\varphi) \geq -\delta) = 1$ where
\begin{equs}
\cI^N(\varphi) &:=\int_0^\infty\!\!\! \bigg(\Big\langle \partial_s \varphi(s,\cdot) , \big( \eta^N_s(\cdot) - \zeta^N_s(\cdot) \big)^\pm \Big\rangle_N \!\!\!+ \Big\langle \partial_x \varphi(s,\cdot),H^\pm\big(\tau_\cdot \eta^N_s,\tau_\cdot \zeta^N_s \big) \Big\rangle_N \\
&+ 2 \Big( (1-c)^\pm \varphi(s,0) + (0-c)^\pm \varphi(s,1) \Big)\bigg) ds+ \Big\langle \varphi(0,\cdot) , \big( \eta^N_0(\cdot) - \zeta^N_0(\cdot) \big)^\pm \Big\rangle_N\;,
\end{equs}
where $H^+(\eta,\zeta) = -2 \big( b(\eta(1),\eta(0))-b(\zeta(1),\zeta(0))\big) F_{1,0}(\eta,\zeta)$ and $H^-(\eta,\zeta)=H^+(\zeta,\eta)$.
\end{lemma}
This is an adaptation of Theorem 3.1 in~\cite{Reza}.
\begin{proof}
We define
\begin{equs}
B_t &=\int_0^t \Big( \big\langle \partial_s\varphi(s,\cdot) , \big(\eta_s(\cdot)-\zeta_s(\cdot)\big)^\pm \big\rangle_N + \tilde{\cL}\big\langle \varphi(s,\cdot) , \big(\eta_s(\cdot)-\zeta_s(\cdot)\big)^\pm \big\rangle_N \Big)ds\\
&\quad+ \Big\langle \varphi(0,\cdot) , \big( \eta^N_0(\cdot) - \zeta^N_0(\cdot) \big)^\pm \Big\rangle_N\;.
\end{equs}
We have the identity
\begin{equ}\label{Eq:Dynkin}
\Big\langle \varphi(t,\cdot) , \big( \eta^N_t(\cdot) - \zeta^N_t(\cdot) \big)^\pm \Big\rangle_N = B_t + M_t\;,
\end{equ}
where $M$ is a mean zero martingale. Since $\varphi$ has compact support, the l.h.s.~vanishes for $t$ large enough. Below, we work at an arbitrary time $s$ so we drop the subscript $s$ in the calculations. Moreover, we write $\varphi(k)$ instead of $\varphi(k/2N)$ to simplify notations. We treat separately the boundary part and the bulk part of the generator. Regarding the former, we have
\begin{equs}
{}&\tilde{\cL}^{\mbox{\tiny bdry}}\big\langle \varphi(\cdot) , \big(\eta(\cdot)-\zeta(\cdot)\big)^+ \big\rangle_N\\
&= (2N)^\alpha (2p_N-1) \Big(\varphi(1)\eta(1)\zeta(1)(1-c) - \varphi(2N)\eta(2N) (1-\zeta(2N)) c\Big)\\
&\leq 2 \varphi(0) (1-c) + \cO(N^{-\alpha})\;,
\end{equs}
since $\varphi$ is non-negative and $2p_N-1 \sim 2(2N)^{-\alpha}$. Similarly, we find
\begin{equ}
\tilde{\cL}^{\mbox{\tiny bdry}}\big\langle \varphi(\cdot) , \big(\eta(\cdot)-\zeta(\cdot)\big)^- \big\rangle_N\leq 2 \varphi(2N) (0-c)^- + \cO(N^{-\alpha})\;.
\end{equ}

We turn to the bulk part of the generator. Recall the map $F_{k,\ell}(\eta,\zeta)$, and set $G_{k,\ell}(\eta,\zeta) = 1 - F_{k,\ell}(\eta,\zeta)F_{k,\ell}(\zeta,\eta)$. By checking all the possible cases, one easily gets the following identity
\begin{equs}
\tilde{\cL}^{\mbox{\tiny bulk}}\big(\eta(k)-\zeta(k)\big)^+ &= (2N)^{1+\alpha}\sum_\ell \bigg[ \Big(p(\ell-k)\big(b(\zeta(k),\zeta(\ell))-b(\eta(k),\eta(\ell))\big)\\
&\qquad- p(k-\ell)\big(b(\zeta(\ell),\zeta(k))-b(\eta(\ell),\eta(k))\big) \Big) F_{k,\ell}(\eta,\zeta)\\
&-\Big(p(\ell-k)b(\eta(k),\eta(\ell)) + p(k-\ell)b(\zeta(\ell),\zeta(k))\Big)G_{k,\ell}(\eta,\zeta) \bigg]\;.
\end{equs}
Since $\eta$ and $\zeta$ play symmetric r\^oles in $\tilde{\cL}^{\mbox{\tiny bulk}}$, we find a similar identity for $\tilde{\cL}^{\mbox{\tiny bulk}}\big(\eta(k)-\zeta(k)\big)^-$. Notice that the term on the third line is non-positive, so we will drop it in the inequalities below. We thus get
\begin{equ}
\tilde{\cL}^{\mbox{\tiny bulk}}\big\langle \varphi(\cdot) , \big(\eta(\cdot)-\zeta(\cdot)\big)^\pm \big\rangle_N \leq (2N)^{\alpha} \sum_{\substack{k,\ell= 1\\ \ell = k \pm 1}}^{2N} p(\ell-k) \big(\varphi(k)-\varphi(\ell)\big) I_{k,\ell}^\pm(\eta,\zeta)\;,
\end{equ}
where
\begin{equ}
I_{k,\ell}^+(\eta,\zeta)= \big(b(\zeta(k),\zeta(\ell))-b(\eta(k),\eta(\ell))\big) F_{k,\ell}(\eta,\zeta)\;,\quad I_{k,\ell}^-(\eta,\zeta)= I_{k,\ell}^+(\zeta,\eta) \;.
\end{equ}
Up to now, we essentially followed the calculations made in the first step of the proof of~\cite[Thm 3.1]{Reza}. At this point, we argue differently: we decompose $p(\pm 1)$ into the symmetric part $1-p_N$, which is of order $1/2$, and the asymmetric part which is either $0$ or $2p_N-1\sim 2 (2N)^{-\alpha}$.\\
We start with the contribution of the symmetric part. Recall the definition of the number of sign changes $n$ and of the integers $k_1 < \ldots < k_{n+1}$. Using a discrete integration by parts, one easily deduces that for all $i \leq n$
\begin{equs}
\sum_{\substack{k,\ell= k_i\\ \ell = k \pm 1}}^{k_{i+1}-1} \big(\varphi(k)-\varphi(\ell)\big) I_{k,\ell}^\pm(\eta,\zeta) &= \sum_{k=k_{i}}^{k_{i+1}-2}\big(\eta(k)-\zeta(k)\big)^\pm \Delta \varphi(k)\\
&- \big(\eta(k_{i+1}-1)-\zeta(k_{i+1}-1)\big)^\pm \nabla \varphi(k_{i+1}-2)\\
&+ \big(\eta(k_{i})-\zeta(k_{i})\big)^\pm \nabla \varphi(k_{i}-1)\;.
\end{equs}
Since $n(s)$ is bounded uniformly over all $N\geq 1$ and all $s\geq 0$, we deduce that the boundary terms arising at the second and third lines yield a negligible contribution. Thus we find
\begin{equs}
(2N)^\alpha \sum_{\substack{k,\ell= 1\\ \ell = k \pm 1}}^{2N} (1-p_N)  \big(\varphi(k)-\varphi(\ell)\big) I_{k,\ell}^\pm(\eta,\zeta) = \cO\Big(\frac1{N^{1-\alpha}}\Big)\;.
\end{equs}
Regarding the asymmetric part $p(\pm 1)-(1-p_N)$, a simple calculation yields the identity
\begin{equs}
{}&(2N)^{\alpha} \sum_{\substack{k,\ell= 1\\ \ell = k \pm 1}}^{2N} \big(p(\ell-k)-1+p_N\big) \big(\varphi(k)-\varphi(\ell)\big) I_{k,\ell}^\pm(\eta,\zeta)\\
&= \frac1{2N}\sum_{k=1}^{2N-1} \partial_x \varphi(k) \tau_k H^\pm(\eta,\zeta) + \cO(N^{-\alpha})\;,
\end{equs}
uniformly over all $N\geq 1$. Therefore
\begin{equs}
\tilde{\cL}^{\mbox{\tiny bulk}}\big\langle \varphi(\cdot) , \big(\eta(\cdot)-\zeta(\cdot)\big)^\pm \big\rangle_N \leq \frac1{2N}\sum_{k=1}^{2N-1} \partial_x\varphi(k) \tau_k H^\pm(\eta,\zeta) + \cO\Big(\frac1{N^{\alpha\wedge(1-\alpha)}}\Big)\;.
\end{equs}
Putting together the two contributions of the generator, we get
\begin{equs}
B_t &\leq \int_0^t \Big( \big\langle  \partial_s\varphi(s,\cdot) , \big(\eta^N_s(\cdot)-\zeta^N_s(\cdot)\big)^\pm \big\rangle_N + \big\langle \partial_x \varphi(s,\cdot) , \tau_\cdot H^\pm(\eta^N_s,\zeta^N_s) \big\rangle_N  \\
&\quad+ 2t \big( (1-c)^\pm \varphi(s,0) + (0-c)^\pm \varphi(s,1) \big) \Big)ds\\
&\quad+ \Big\langle \varphi(0,\cdot) , \big( \eta^N_0(\cdot) - \zeta^N_0(\cdot) \big)^\pm \Big\rangle_N +\cO\Big(\frac1{N^{\alpha\wedge(1-\alpha)}}\Big)\;.
\end{equs}
Recall the equation (\ref{Eq:Dynkin}). A simple calculation shows that $\tilde{\E}^N_{\iota_N,c} \langle M \rangle_t \lesssim \frac{1}{N^{1-\alpha}}$ uniformly over all $N\geq 1$ and all $t\geq 0$. Moreover, the jumps of $M$ are almost surely bounded by a term of order $N^{-1}$. Applying the BDG inequality (\ref{Eq:BDG3}), we deduce that
\begin{equ}
\tilde{\E}^N_{\iota_N,c} \Big[\sup_{s\leq t} M_s^2\Big]^\frac12 \lesssim \frac{1}{N^{\frac{1-\alpha}{2}}}\;,
\end{equ}
uniformly over all $N\geq 1$ and all $t\geq 0$. Since $\varphi$ has compact support, $B_t = -M_t$ for $t$ large enough. The assertion of the lemma then easily follows.
\end{proof}
Recall that $\ccM_{T_\ell(u)} \eta$ is the average of $\eta$ on the box $T_\ell(u)$ for any $u\in \{1,\ldots,2N\}$.
\begin{lemma}[Macroscopic inequalities]\label{Lemma:MacroIneq}
Let $\iota_N$ be a measure on $\{0,1\}^{2N}$ satisfying Assumption \ref{Assumption:IC}. For all $\varphi\in \cC^\infty_c([0,\infty)\times[0,1],\R_+)$, all $\delta > 0$ and all $c\in [0,1]$, we have $\lim_{\epsilon \downarrow 0} \varliminf_{N\rightarrow\infty} \bbP^N_{\iota_N}(\cJ^N(\varphi) \geq -\delta) = 1$ where
\begin{equs}[Eq:ClaimHydro]
\cJ^N(\varphi) &:= \int_0^\infty \bigg(\Big\langle \partial_s \varphi(s,\cdot) , \Big( \ccM_{T_{\epsilon N}(\cdot)}(\eta^N_s) - c \Big)^\pm \Big\rangle_N+ \Big\langle \partial_x \varphi(s,\cdot),h^\pm\Big(\ccM_{T_{\epsilon N}(\cdot)}(\eta^N_s),c\Big) \Big\rangle_N \\
&+ 2 \Big( (1-c)^\pm \varphi(s,0) + (0-c)^\pm \varphi(s,1) \Big)\bigg) ds\\
&+ \Big\langle \varphi(0,\cdot) , \Big( \ccM_{T_{\epsilon N}(\cdot)}(\eta^N_0) - c \Big)^\pm \Big\rangle_N\;.
\end{equs}
\end{lemma}
\begin{proof}
Since at any time $s\geq 0$, $\zeta^N(s,\cdot)$ is distributed according to a product of Bernoulli measures with parameter $c$, we deduce that
\begin{equ}
\lim_{\epsilon \downarrow 0} \lim_{N\rightarrow\infty} \tilde{\E}^N_{\iota_N,c}\Big[\frac1{2N} \sum_{u=1}^{2N} \Big| \ccM_{T_{\epsilon N}(u)}(\zeta^N_s) - c\Big|\Big] = 0\;.
\end{equ}
and consequently, by Fubini's Theorem and stationarity, we have
\begin{equ}
\lim_{\epsilon \downarrow 0} \lim_{N\rightarrow\infty} \tilde{\E}^N_{\iota_N,c}\Big[\int_0^t \frac1{2N} \sum_{u=1}^{2N} \Big| \ccM_{T_{\epsilon N}(u)}(\zeta^N_s) - c\Big| ds \Big] = 0\;.
\end{equ}
Now we observe that for all $\epsilon >0$, we have $\tilde{\P}^N_{\iota_N,c}$ almost surely
\begin{equ}
\Big\langle \varphi(0,\cdot) , \big( \eta^N_0(\cdot) - \zeta^N_0(\cdot) \big)^\pm \Big\rangle_N = \Big\langle \varphi(0,\cdot) , \ccM_{T_{\epsilon N}(\cdot)}\big( \eta^N_0 - \zeta^N_0 \big)^\pm \Big\rangle_N + \cO(\epsilon)\;.
\end{equ}
Recall the coupling we chose for $(\eta^N_0(\cdot),\zeta^N_0(\cdot))$. Since $\tilde{\P}^N_{\iota_N,c}$ almost surely the number of sign changes $n(0)$ is bounded by some constant $C>0$ uniformly over all $N\geq 1$, we deduce using the previous identity that
\begin{equ}
\Big\langle \varphi(0,\cdot) , \big( \eta^N_0(\cdot) - \zeta^N_0(\cdot) \big)^\pm \Big\rangle_N = \Big\langle \varphi(0,\cdot) , \big( \ccM_{T_{\epsilon N}(\cdot)}\eta^N_0 - \ccM_{T_{\epsilon N}(\cdot)}\zeta^N_0 \big)^\pm \Big\rangle_N + \cO(\epsilon)\;.
\end{equ}
Therefore, by Lemma \ref{Lemma:MicroIneq}, we deduce that the statement of the lemma follows if we can show that for all $\delta > 0$
\begin{equs}[Eq:ClaimHydro2]
\varlimsup_{\epsilon \downarrow 0} \varlimsup_{N\rightarrow\infty} \tilde{\P}^N_{\iota_N,c}&\bigg( \int_0^t \frac1{2N} \sum_{u=1}^{2N} \Big| \ccM_{T_{\epsilon N}(u)} \big( \eta^N_s - \zeta^N_s \big)^\pm\\
&\qquad\qquad - \big( \ccM_{T_{\epsilon N}(u)}(\eta^N_s-\zeta^N_s) \big)^\pm \Big| ds > \delta \bigg) = 0\;,\\
\varlimsup_{\epsilon \downarrow 0} \varlimsup_{N\rightarrow\infty} \tilde{\P}^N_{\iota_N,c}&\bigg( \int_0^t \frac1{2N} \sum_{u=1}^{2N} \Big| \ccM_{T_{\epsilon N}(u)} H^\pm(\eta^N_s,\zeta^N_s) \\
&\qquad\qquad- h^\pm\Big( \ccM_{T_{\epsilon N}(u)}(\eta^N_s), \ccM_{T_{\epsilon N}(u)}(\zeta^N_s)\Big) \Big| ds > \delta \bigg) = 0\;.
\end{equs}
We restrict ourselves to proving the second identity, since the first is simpler. Let $\cN^+_s$, resp. $\cN^-_s$, be the set of $u\in\{1,\ldots,2N\}$ such that $\eta_s \geq \zeta_s$, resp. $\zeta_s \geq \eta_s$, on the whole box $T_{\epsilon N}(u)$. By Lemma \ref{Lemma:Coupling}, $2N-\#\cN^+_s-\#\cN^-_s$ is of order $\epsilon N$ uniformly over all $s$, all $N\geq 1$ and all $\epsilon$. Therefore, we can neglect the contribution of all $u \notin \cN^+_s \cup \cN^-_s$. If we define $\Phi(\eta)=-2 \eta(1)(1-\eta(0))$ and if we let $\tilde{\Phi}(a)$ be as in (\ref{Eq:tildePhi}) below, then for all $u\in \cN^+_s$ we have 
\begin{equ}
\ccM_{T_{\epsilon N}(u)} H^-(\eta^N_s,\zeta^N_s) - h^-\Big( \ccM_{T_{\epsilon N}(u)}(\eta^N_s), \ccM_{T_{\epsilon N}(u)}(\zeta^N_s)\Big) = 0\;,
\end{equ}
as well as
\begin{equs}
{}&\ccM_{T_{\epsilon N}(u)} H^+(\eta^N_s,\zeta^N_s) - h^+\Big( \ccM_{T_{\epsilon N}(u)}(\eta^N_s), \ccM_{T_{\epsilon N}(u)}(\zeta^N_s)\Big)\\
= \;&\ccM_{T_{\epsilon N}(u)} \Phi(\eta^N_s) - \tilde{\Phi}\Big(\ccM_{T_{\epsilon N}(u)} \eta^N_s\Big)- \ccM_{T_{\epsilon N}(u)} \Phi(\zeta^N_s) + \tilde{\Phi}\Big(\ccM_{T_{\epsilon N}(u)} \zeta^N_s\Big)\;.
\end{equs}
Similar identities hold for every $u\in \cN^-_s$. We deduce that the second identity of (\ref{Eq:ClaimHydro2}) follows if we can show that for all $\delta >0$
\begin{equs}
\varlimsup_{\epsilon\downarrow 0} \varlimsup_{N\rightarrow\infty} \P^N_{\iota_N}\bigg( \int_0^t \frac1{2N} \sum_{u=1}^{2N} \Big| \ccM_{T_{\epsilon N}(u)} \Phi(\eta^N_s) - \tilde{\Phi}\Big(\ccM_{T_{\epsilon N}(u)}\eta_s\Big) \Big| ds > \delta \bigg) &= 0\;,\\
\varlimsup_{\epsilon\downarrow 0} \varlimsup_{N\rightarrow\infty} \tilde{\E}^N_{\iota_N,c}\bigg[ \int_0^t \frac1{2N} \sum_{u=1}^{2N} \Big| \ccM_{T_{\epsilon N}(u)}\Phi(\zeta^N_s) - \tilde{\Phi}\Big(\ccM_{T_{\epsilon N}(u)}\zeta^N_s\Big) \Big| ds \bigg] &= 0\;.
\end{equs}
The first convergence is ensured by Theorem \ref{Th:Replacement}, while the second follows from the stationarity of $\zeta^N$ and the Ergodic Theorem. This completes the proof of the lemma.
\end{proof}
\begin{proof}[of Theorem \ref{Th:HydroProd}]
For any given $\epsilon > 0$, we have
\begin{equs}[Eq:AverageEta]
\ccM_{T_{2\epsilon N}(k)}(\eta_s) &= \frac1{2\epsilon} \rho^N\Big(s,\Big[\frac{k}{2N}-\epsilon,\frac{k}{2N}+\epsilon\Big]\Big)\\
&= \frac1{2\epsilon} \rho^N\Big(s,[x-\epsilon,x+\epsilon]\Big) + \cO(N^{-1})\;,
\end{equs}
uniformly over all $k\in\{1,\ldots,2N-1\}$, all $x\in \Big[\frac{k}{2N},\frac{k+1}{2N}\Big]$ and all $N\geq 1$. Notice that the $\cO(N^{-1})$ depends on $\epsilon$. For all $\rho\in \bbD\big([0,\infty),\cM([0,1])\big)$, we set
\begin{equs}
V_c(\epsilon,\rho) &:= \int_0^\infty \bigg(\Big\langle \partial_s \varphi(s,\cdot) , \Big( \frac1{2\epsilon} \rho\Big(s,\Big[\cdot-\epsilon,\cdot+\epsilon\Big]\Big) - c \Big)^\pm \Big\rangle \\
&\quad+ \Big\langle \partial_x \varphi(s,\cdot),h^\pm\Big(\frac1{2\epsilon} \rho\Big(s,\Big[\cdot-\epsilon,\cdot+\epsilon\Big]\Big),c\Big) \Big\rangle \\
&\quad+ 2 \Big( (1-c)^\pm \varphi(s,0) + (0-c)^\pm \varphi(s,1) \Big)\bigg) ds\\
&\quad+\Big \langle \varphi(0,\cdot), \Big( \frac1{2\epsilon} \rho\Big(0,\Big[\cdot-\epsilon,\cdot+\epsilon\Big]\Big) - c \Big)^\pm \Big\rangle\;.
\end{equs}
Combining (\ref{Eq:AverageEta}), (\ref{Eq:ClaimHydro}) and the continuity of the maps $h^\pm(\cdot,c)$ and $(\cdot)^\pm$, we deduce that for any $\delta > 0$, we have
\begin{equs}
\lim_{\epsilon \downarrow 0} \varliminf_{N\rightarrow\infty} \bbP^N_{\iota_N}&\big( V_c(\epsilon,\rho^N) \geq -\delta\big) = 1\;.
\end{equs}
At this point, we observe that for all $\varphi\in\cC([0,1],\R_+)$ we have
\begin{equ}
\langle \rho^N(t),\varphi \rangle \leq \frac1{2N} \sum_{k=1}^{2N} \varphi(k/2N)\;,
\end{equ}
so that a simple argument ensures that for every limit point $\rho$ of $\rho^N$ and for all $t\geq 0$, the measure $\rho(t,dx)$ is absolutely continuous with respect to the Lebesgue measure, and its density is bounded by $1$. Therefore, any limit point is of the form $\rho(t,dx)=\eta(t,x)dx$ with $\eta\in L^\infty\big([0,\infty)\times(0,1)\big)$. Let $\P$ be the law of the limit of a converging subsequence $\rho^{N_i}$. Since $\rho \mapsto V_c(\epsilon,\rho)$ is a $\P$-a.s.~continuous map on $\bbD\big([0,\infty),\cM([0,1])\big)$, we have for all $\epsilon > 0$
\begin{equ}
\varlimsup_{i\rightarrow\infty} \bbP^{N_i}_{\iota_{N_i}}\big( V_c(\epsilon,\rho^{N_i}) \geq -\delta\big) \leq \bbP\big( V_c(\epsilon,\rho) \geq -\delta\big)\;.
\end{equ}
For any $\rho$ of the form $\rho(t,dx) = \eta(t,x)dx$, we set
\begin{equs}
V_c(\rho) &:= \int_0^\infty \bigg(\Big\langle \partial_s \varphi(s,\cdot) , \Big( \eta(s,\cdot) - c \Big)^\pm \Big\rangle + \Big\langle \partial_x \varphi(s,\cdot),h^\pm\Big(\eta(s,\cdot),c\Big) \Big\rangle \\
&\quad+ 2 \Big( (1-c)^\pm \varphi(s,0) + (0-c)^\pm \varphi(s,1) \Big)\bigg) ds + \Big\langle \varphi(0,\cdot), (\eta_0-c)^\pm \Big\rangle \;,
\end{equs}
and we observe that by Lebesgue Differentiation Theorem, we have $\P$-a.s.~$V_c(\rho) = \lim_{\epsilon\downarrow 0} V_c(\epsilon,\rho)$. Therefore,
\begin{equs}
\P\big(V_c(\rho) \geq -\delta \big) &= \P\big(\lim_{\epsilon\downarrow 0} V_c(\epsilon,\rho) \geq -\delta \big) \geq \E\big[\varlimsup_{\epsilon\downarrow 0} \tun_{\{V_c(\epsilon,\rho) \geq -\delta/2\}}\big]\\
&\geq \varlimsup_{\epsilon\downarrow 0}\E\big[ \tun_{\{V_c(\epsilon,\rho) \geq -\delta/2\}}\big] \geq \varlimsup_{\epsilon\downarrow 0}\varlimsup_{i\rightarrow\infty} \bbP^{N_i}_{\iota_{N_i}}\big( V_c(\epsilon,\rho^{N_i}) \geq -\delta/2\big) =1\;,
\end{equs}
so the process $\big(\eta(t,x),t\geq 0, x\in (0,1)\big)$ under $\P$ coincides with the unique entropy solution of (\ref{PDEBurgersDirichlet}), thus concluding the proof.
\end{proof}

\subsection{The replacement lemma}

Let $r\geq 1$ be an integer and $\Phi:\{0,1\}^r \rightarrow\R$. For all $\eta \in \{0,1\}^{2N}$ and as soon as $2N \geq r$, we define
\begin{equ}
\Phi(\eta) := \Phi\big(\eta(1),\ldots,\eta(r)\big)\;.
\end{equ}
We also introduce the expectation of $\Phi$ under a product of Bernoulli measures with parameter $a \in [0,1]$:
\begin{equ}\label{Eq:tildePhi}
\tilde{\Phi}(a) := \sum_{\eta \in \{0,1\}^r} \Phi(\eta) a^{\#\{i:\eta(i)=1\}} (1-a)^{\#\{i:\eta(i)=0\}}\;.
\end{equ}
We consider the sequence $\Phi(\eta)(k):= \Phi(\tau_k \eta)$ and the associated averages $\ccM_{T_\ell(k)} \Phi(\eta)$. The ``replacement lemma" controls the probability that the following quantity is large
\begin{equ}
V_{\ell}(\eta) = \Big| \ccM_{T_\ell(0)} \Phi(\eta) - \tilde{\Phi}\Big(\ccM_{T_\ell(0)} \eta\Big)\Big|\;.
\end{equ}

\begin{theorem}[Replacement lemma]\label{Th:Replacement}
We work under Assumption \ref{Assumption:IC}. For every $\delta > 0$, we have
\begin{equ}\label{Eq:Replacement}
\varlimsup_{\epsilon\downarrow 0}\varlimsup_{N\rightarrow\infty} \bbP^N_{\iota_N} \Big(\int_0^t \frac{1}{N}\sum_{k=1}^{2N} V_{\epsilon N}(\tau_k\eta_s) ds \geq \delta \Big) = 0\;.
\end{equ}
\end{theorem}
The proof of this theorem relies on the classical one-block and two-blocks estimates, we refer to the companion paper~\cite{LabbeReview} for the complete proof.

\section{KPZ fluctuations}\label{Section:KPZ}

To prove Theorem \ref{Th:KPZ}, we follow the method of Bertini and Giacomin~\cite{BG97}. Due to our boundary conditions, there are two important steps that need some specific arguments: first the bound on the moments of the discrete process, see Proposition \ref{Prop:BoundMomentsKPZ}, second the bound on the error terms arising in the identification of the limit, see Proposition \ref{Prop:DelicateKPZ}. The remaining arguments then follow \textit{mutatis mutandis} so we will only state the intermediate results and omit the proofs. In order to simplify the notations, we will regularly use the microscopic variables $k,\ell \in\{1,\ldots,2N-1\}$ in rescaled quantities: for instance $h^N(t,\ell)$ stands for $h^N(t,x)$ with $x=(\ell-N)/(2N)^{2\alpha}$.\\
The proof relies on the discrete Hopf-Cole transform, which is due to G\"artner~\cite{Gartner88}. Set $\xi^N(t,x) := \exp(-h^N(t,x))$, where $h^N$ was introduced in (\ref{Eq:hN}). The stochastic differential equations solved by $\xi^N$ are given by
\begin{equs}[Eq:SDEKPZ]
	\begin{cases}
	d\xi^N(t,\ell) = c_N \Delta \xi^N(t,\ell)dt + dM^N(t,\ell)\;,\qquad  \ell\in\{1,\ldots,2N-1\}\;,\\
	\xi^N(t,0) = \xi^N(t,2N) = e^{\lambda_N t}\;,\\
	\xi^N(0,\cdot) = e^{-h^N(0,\cdot)} \;,\end{cases}
\end{equs}
where $M^N$ is a martingale with bracket given by $\langle M^N(\cdot,k),M^N(\cdot,\ell)\rangle_t = 0$ whenever $k\ne \ell$, and
\begin{equs}[Eq:Bracket]
	d \langle M^N(\cdot,k)\rangle_t &= \lambda_N\Big(\xi^N(t,k)\Delta\xi^N(t,k)+2\xi^N(t,k)^2 \Big)dt\\
	&\quad- (2N)^{4\alpha} \nabla^+\xi^N(t,k) \nabla^-\xi^N(t,k)dt\;,
\end{equs}
where we rely on the notation
\begin{equ}
\nabla^+ f(\ell):=f(\ell+1)-f(\ell)\;,\quad \nabla^- f(\ell) := f(\ell)-f(\ell-1)\;.
\end{equ}
Observe that
\begin{equ}
\big| d \langle M^N(\cdot,k)\rangle_t \big| \lesssim \xi^N(t,k)^2 (2N)^{2\alpha} dt\;,
\end{equ}
uniformly over all $t\geq 0$, all $k$ and all $N\geq 1$. As usual, we let $\cF_t,t\geq 0$ be the natural filtration associated with the process $(\xi^N(t),t\geq 0)$.

We define $B_0^N(t) := [\frac{\lambda_N}{\gamma_N}t,2N-\frac{\lambda_N}{\gamma_N}t] \subset[0,2N]$. The hydrodynamic limit obtained in Theorem \ref{Th:Hydro} shows that this is the window where the density of particles is approximately $1/2$ at time $t$ (in the time scale $(2N)^{4\alpha}$). On the left of this window, the density is approximately $1$, and on the right it is approximately $0$. For technical reasons, it is convenient to introduce an $\epsilon$-approximation of this window by setting:
\begin{equ}
B^N_\epsilon(t) := \Big[\frac{\lambda_N}{\gamma_N} t + \epsilon N, 2N-\frac{\lambda_N}{\gamma_N} t - \epsilon N \Big]\;,\quad t\in [0,T)\;.
\end{equ}

We let $p^N_t(k,\ell)$ be the discrete heat kernel on $\{0,\ldots,2N\}$ sped up by $2c_N$ and endowed with homogeneous Dirichlet boundary conditions: we refer to Appendix \ref{Appendix:Kernel} for a definition and some properties. Classical arguments ensure that the unique solution of (\ref{Eq:SDEKPZ}) is given by
\begin{equs}[Eq:DiscreteHC]
	\xi^N(t,\ell) &= I^N(t,\ell) + N^t_t(\ell)\;,
\end{equs}
where, for all $t\geq 0$, $[0,t]\ni r \mapsto N^t_r(\ell)$ is the martingale
\begin{equ}\label{Eq:DefNt}
	N^t_r(\ell) = \int_0^r \sum_{k=1}^{2N-1} p^N_{t-s}(k,\ell) dM^N(s,k)\;,
\end{equ}
and $I^N(t,\ell)$ is the term coming from the initial-boundary conditions. More precisely, we write $I^N(t,\ell) = \xi^{N,\circ}(t,\ell) + \sum_{k=1}^{2N-1} p^N_t(k,\ell)\big(\xi^N(0,k)-1\big)$ where $\xi^{N,\circ}$ is the solution of
\begin{equs}
	\begin{cases}\partial_t\xi^{N,\circ}(t,\ell) = c_N \Delta \xi^{N,\circ}(t,\ell)\;,\\
	\xi^{N,\circ}(t,0) = \xi^{N,\circ}(t,2N) = e^{\lambda_N t}\;,\\
	\xi^{N,\circ}(0,\ell) = 1\;.\end{cases}
\end{equs}
Observe that $\xi^N$ and $\xi^{N,\circ}$ satisfy the same inhomogeneous Dirichlet boundary conditions. Consequently, all the other terms satisfy homogeneous Dirichlet boundary conditions, so that we can rely on the corresponding heat kernel in their expressions.\\

We define
\begin{equ}
	b^N(t,\ell) := 2 + \exp\Big(\lambda_N t - \gamma_N \big(\ell \wedge (2N - \ell)\big)\Big)\;.
\end{equ}

\begin{remark}
The hydrodynamic limit of Theorem \ref{Th:Hydro}, upon Hopf-Cole transform, is given by $1\vee \exp\Big(\lambda_N t - \gamma_N \big(\ell \wedge (2N - \ell)\big)\Big)$.
\end{remark}

\begin{proposition}\label{Prop:IC}
Let $K$ be a compact subset of $[0,T)$ and fix $\epsilon >0$. Uniformly over all $t\in K$, we have
\begin{itemize}
\item $I^N(t,\ell) \lesssim b^N(t,\ell)$ for all $\ell\in\{1,\ldots,2N\}$,
\item $\big|\nabla^\pm I^N(t,\ell)\big| \lesssim t^{-\frac12} N^{-3\alpha}$ uniformly over all $\ell \in B^N_\epsilon(t)$.
\end{itemize}
\end{proposition}
\begin{proof}
Since our initial condition is flat, it is immediate to check that
\begin{equ}
\Big|I^N(t,\ell) - \xi^{N,\circ}(t,\ell)\Big| \lesssim N^{-\alpha} \ll b^N(t,\ell)\;.
\end{equ}
Furthermore, using Lemmas \ref{Lemma:ExpoDecay} and \ref{Lemma:DecaySeriesKernel}, we get
\begin{equ}
\nabla^\pm\big(I^N(t,\ell) - \xi^{N,\circ}(t,\ell)\big) = \sum_{k\in B^N_{\epsilon/2}(0)}\!\! \nabla^\pm \bar{p}^N_t(\ell-k)\big(\xi^{N}(0,k)-1\big)+ \cO(N^{1-\alpha} e^{-\delta N^{2\alpha}})\;,
\end{equ}
uniformly over all $\ell\in B^N_\epsilon(t)$, all $t\in K$ and all $N\geq 1$. Then, we write
\begin{equ}
\sum_{k\in B^N_{\epsilon/2}(0)} |\nabla^+ \bar{p}^N_t(\ell-k)| = -\bar{p}^N_t(\ell-i_- -1) + 2\bar{p}^N_t(0) - \bar{p}^N_t(\ell-i_+)\;,
\end{equ}
where $i_{\pm}$ are the first and last integers in $B^N_{\epsilon/2}(0)$, and $\bar{p}^N$ is the discrete heat kernel on $\Z$, see Appendix \ref{Appendix:Kernel}. Using Lemma \ref{Lemma:BoundHeatKernelZ} and the fact that the initial condition is flat, we deduce that
\begin{equ}
\Big|\sum_{k\in B^N_{\epsilon/2}(0)} \nabla^+ \bar{p}^N_t(\ell-k)(\xi^{N}(0,k)-1)\Big| \lesssim \frac{1}{\sqrt{t}(2N)^{3\alpha}}\;,
\end{equ}
uniformly over the same set of parameters, as required. The same arguments work for $\nabla^-$.\\
To establish the required bounds on $\xi^{N,\circ}$, we first observe that we have
\begin{equ}
\xi^{N,\circ}(t,\ell) = 1 + \lambda_N \int_0^t \Big(1-\sum_{k=1}^{2N-1} p^N_{t-s}(k,\ell) \Big) e^{\lambda_N s} ds\;.
\end{equ}
In Appendix \ref{Appendix:Kernel}, we show that there exists $\delta > 0$ such that
\begin{equ}\label{Eq:BoundTailDiscreteKernel}
\Big(1-\sum_{k=1}^{2N-1} p^N_{t-s}(k,\ell) \Big) e^{\lambda_N s} \lesssim e^{-\delta N^{2\alpha}}\;,
\end{equ}
uniformly over all $s\in[0,t]$, all $t\in K$ and all $\ell \in B^N_\epsilon(t)$. Consequently, there exists $\delta' > 0$ such that $|\xi^{N,\circ}(t,\ell) - 1| \lesssim \exp(-\delta' N^{2\alpha})$ uniformly over all $t\in K$, all $\ell\in B^N_\epsilon(t)$ and all $N\geq 1$. Using this bound, we immediately get the bound $\big|\nabla^\pm \xi^{N,\circ}(t,\ell)\big| \lesssim t^{-\frac12} N^{-3\alpha}$ as required. Furthermore, we deduce that for $N$ large enough, $b^N$ solves
\begin{equs}
	\begin{cases}
	\partial_t b^N(t,\ell) = c_N \Delta b^N(t,\ell)\;,\quad \ell\in \{1,\ldots,N-1\}\;,\\
	b^N(t,0) \geq \xi^{N,\circ}(t,0)\;,\quad b^N(t,N) \geq \xi^{N,\circ}(t,N)\;,\\
	b^N(0,k) \geq \xi^{N,\circ}(0,k)\;.\end{cases}
\end{equs}
By the maximum principle, one deduces that $b^N(t,\ell) \geq \xi^{N,\circ}(t,\ell)$ for all $t\in K$ and all $\ell\in\{0,\ldots,N\}$. By symmetry, this inequality also holds for $\ell \in \{N,\ldots,2N\}$.
\end{proof}
To alleviate the notation, we define
\begin{equ}\label{Def:qN}
	q^N_{s,t}(k,\ell) = p^N_{t-s}(k,\ell) b^N(s,k)\;.
\end{equ}
We now have all the ingredients at hand to bound the moments of $\xi^N$.

\begin{proposition}\label{Prop:BoundMomentsKPZ}
For all $n\geq 1$ and all compact set $K \subset [0,T)$, we have
\begin{equ}
	\sup_{N\geq 1} \sup_{\ell\in\{1,\ldots,2N-1\}} \sup_{t\in K} \E\Big[ \Big(\frac{\xi^N(t,\ell)}{b^N(t,\ell)}\Big)^n \Big] < \infty\;.
\end{equ}
\end{proposition}
Since $b^N$ is of order $1$ inside $B^N_\epsilon(t)$, this ensures that the moments are themselves of order $1$ in these windows.

\begin{proof}
We fix the compact set $K$ until the end of the proof. Using the expression (\ref{Eq:DiscreteHC}) and Proposition \ref{Prop:IC}, we deduce that
\begin{equ}\label{Eq:ExpresMoments}
	\E\bigg[ \bigg(\frac{\xi^N(t,\ell)}{b^N(t,\ell)}\bigg)^{2n} \bigg]^\frac{1}{2n} \lesssim 1 + \E\bigg[ \bigg(\frac{N^t_t(\ell)}{b^N(t,\ell)}\bigg)^{2n} \bigg]^\frac{1}{2n}\;.
\end{equ}
We set $D^t_r := \big[ N^t_\cdot \big]_r - \langle N^t_\cdot \rangle_r$ and we refer to Appendix \ref{Appendix:Mgale} for the notations. By the BDG inequality (\ref{Eq:BDG2}), we obtain
\begin{equ}\label{Eq:BDGMoments}
	\E \Big[ \big(N^t_t(\ell)\big)^{2n}\Big] \lesssim \E \Big[ \big\langle N^t_\cdot(\ell)\big\rangle_t^n\Big] + \E \Big[ \big[ D^t_\cdot(\ell)\big]_t^\frac{n}{2}\Big]\;,
\end{equ}
uniformly over all $\ell \in \{1,\ldots,2N-1\}$, all $t \geq 0$, and all $N\geq 1$. Let
\begin{equ}
	g^N_n(s) := \sup_{k\in\{1,\ldots,2N-1\}} \E\bigg[ \Big(\frac{\xi^N(s,k)}{b^N(s,k)}\Big)^{2n} \bigg]\;.
\end{equ}
We claim that
\begin{equs}
	\E \Big[ \big\langle N^t_\cdot(\ell)\big\rangle_t^n\Big] &\lesssim b^N(t,\ell)^{2n}\int_0^t  \frac{g_n^N(s)}{\sqrt{t-s}} ds\;,\label{Eq:BoundBracket}\\
	\E \Big[ \big[ D^t_\cdot(\ell)\big]_t^\frac{n}{2}\Big] &\lesssim b^N(t,\ell)^{2n}\Big(1 + \int_0^t  \frac{g_n^N(s)}{\sqrt{t-s}} ds\Big)\;,\quad\label{Eq:BoundQuadVar}
\end{equs}
uniformly over all $\ell\in\{1,\ldots,2N-1\}$, all $N\geq 1$ and all $t\in K$. We postpone the proof of these two bounds. Combining these two bounds with (\ref{Eq:ExpresMoments}) and (\ref{Eq:BDGMoments}), we obtain the following closed inequality
\begin{equ}
	g^N_n(t) \lesssim 1 + \int_0^t \frac{g^N_n(s)}{\sqrt{t-s}} ds\;,
\end{equ}
uniformly over all $N\geq 1$ and all $t\in K$. By a generalised Gr\"onwall's inequality, see for instance~\cite[Lemma 6 p.33]{Haraux}, we deduce that $g^N_n(t)$ is uniformly bounded over all $N\geq 1$ and all $t\in K$.\\
We are left with establishing (\ref{Eq:BoundBracket}) and (\ref{Eq:BoundQuadVar}). Using (\ref{Eq:Bracket}), we obtain the almost sure bound
\begin{equ}
	\big\langle N^t_\cdot(\ell)\big\rangle_t \lesssim (2N)^{2\alpha}\int_0^t \sum_{k} p^N_{t-s}(k,\ell)^2 \xi^N(s,k)^2 ds\;,
\end{equ}
uniformly over all $N\geq 1$, $t\geq 0$ and $\ell\in\{1,\ldots,2N-1\}$. Recall the function $q^N$ from (\ref{Def:qN}). Using H\"older's inequality at the second line, we find
\begin{equs}
	\E \bigg[ \Big(\frac{\big\langle N^t_\cdot(\ell)\big\rangle_t}{b^N(t,\ell)^2}\Big)^n\bigg] &\lesssim \int\limits_{s_1,\ldots,s_n=0}^t \sum_{k_1,\ldots,k_n} \E\bigg[\prod_{i=1}^n (2N)^{2\alpha} \Big(\frac{q^N_{s_i,t}(k_i,\ell)}{b^N(t,\ell)}\Big)^2 \Big(\frac{\xi^N(s_i,k_i)}{b^N(s_i,k_i)}\Big)^2 \bigg] ds_i\\
	&\lesssim \int\limits_{s_1,\ldots,s_n=0}^t \sum_{k_1,\ldots,k_n} \prod_{i=1}^n (2N)^{2\alpha} \Big(\frac{q^N_{s_i,t}(k_i,\ell)}{b^N(t,\ell)}\Big)^2  g_n^N(s_i)^{\frac{1}{n}} ds_i\\
	&\lesssim \Big(\int\limits_{s=0}^t \sum_{k} (2N)^{2\alpha} \Big(\frac{q^N_{s,t}(k,\ell)}{b^N(t,\ell)}\Big)^2  g_n^N(s)^{\frac{1}{n}} ds \Big)^n\;.
\end{equs}
By the first inequality of Lemma \ref{Lemma:HeatKernel} we bound $\sum_k q^N_{s,t}(k,\ell) / b^N(t,\ell)$ by a term of order $1$, and by the second inequality of the same lemma we bound $(2N)^{2\alpha} \sup_k q^N_{s,t}(k,\ell) / b^N(t,\ell)$ by a term of order $1/\sqrt{t-s}$. Using Jensen's inequality at the second step, we thus get
\begin{equs}
	\E \bigg[ \Big(\frac{\big\langle N^t_\cdot(\ell)\big\rangle_t}{b^N(t,\ell)^2}\Big)^n\bigg] &\lesssim \Big(\int\limits_{s=0}^t \frac{g_n^N(s)^{\frac{1}{n}}}{\sqrt{t-s}} ds \Big)^n\lesssim \int_{0}^t \frac{g_n(s)}{\sqrt{t-s}} ds\;,
\end{equs}
uniformly over all $N\geq 1$, all $t\in K$ and all $\ell\in\{1,\ldots,2N-1\}$, thus yielding (\ref{Eq:BoundBracket}).\\
We turn to the quadratic variation. Let $J_k$ be the set of jump times of $\xi^N(\cdot,k)$. We start with the following simple bound
\begin{equs}
	\big[ D^t_\cdot(\ell)\big]_t &= \sum_{\tau \leq t} \sum_{k} p^N_{t-\tau}(k,\ell)^4\big(\xi^N(\tau,k)-\xi^N(\tau-,k)\big)^4\\
	&\lesssim \gamma_N^4 \sum_{k}\sum_{\tau \leq t; \tau \in J_k} q^N_{\tau,t}(k,\ell)^4\Big(\frac{\xi^N(\tau,k)}{b^N(\tau,k)}\Big)^4\;,
\end{equs}
uniformly over all $N\geq 1$, all $t\geq 0$ and all $\ell\in\{1,\ldots,2N-1\}$. We set $t_i:= i(2N)^{-4\alpha}$ for all $i\in\N$ and we let $\cI_i:=[t_i,t_{i+1})$. Then, by Minkowski's inequality we have
\begin{equ}
	\E \Big[ \big[ D^t_\cdot(\ell)\big]_t^{\frac{n}{2}}\Big]^\frac{2}{n} \lesssim \gamma_N^4\!\!\!\sum_{i=0}^{\lfloor t(2N)^{4\alpha}\rfloor}\!\!\!\sum_{k}\sup_{s\in\cI_i, s<t} q^N_{s,t}(k,\ell)^4\E \Big[ \Big( \sum_{\tau \in \cI_i \cap J_k} \Big(\frac{\xi^N(\tau,k)}{b^N(\tau,k)}\Big)^4\Big)^\frac{n}{2}\Big]^{\frac{2}{n}}\;,
\end{equ}
Let $Q(k,r,s)$ be the number of jumps of the process $\xi^N(\cdot,k)$ on the time interval $[r,s]$. We have the following almost sure bound
\begin{equ}
\xi^N(\tau,k) \leq \xi^N(s,k) e^{2(2N)^{-4\alpha} \lambda_N + 2\gamma_N Q(k,s,t_{i+1})}\;,
\end{equ}
uniformly over all $s\in \cI_{i-1}$, all $\tau \in \cI_i$, all $k\in\{1,\ldots,2N-1\}$ and all $i\geq 1$. Consequently we get
\begin{equ}
	\sum_{\tau \in \cI_i\cap J_k}\Big(\frac{\xi^N(\tau,k)}{b^N(\tau,k)}\Big)^4 \lesssim (2N)^{4\alpha}\int_{t_{i-1}}^{t_i} \Big(\frac{\xi^N(s,k)}{b^N(s,k)}\Big)^4 Q(k,s,t_{i+1})\, e^{8\gamma_N Q(k,s,t_{i+1})} ds\;,
\end{equ}
uniformly over all $N\geq 1$, all $i\geq 1$ and all $k\in\{1,\ldots,2N-1\}$. Since $(Q(k,s,t), t\geq s)$ is, conditionally given $\cF_s$, stochastically bounded by a Poisson process with rate $(2N)^{4\alpha}$, we deduce that there exists $C >0$ such that almost surely
\begin{equ}
	\sup_{N\geq 1} \sup_{i\geq 1} \sup_{s\in\cI_{i-1}} \E\Big[Q(k,s,t_{i+1})^{\frac{n}{2}}e^{4n\gamma_N Q(k,s,t_{i+1})}\,\Big|\,\cF_s\Big] < C\;.
\end{equ}
Then, we get
\begin{equs}
	{}&\E \Big[ \Big( \sum_{\tau \in \cI_i\cap J_k} \Big(\frac{\xi^N(\tau,k)}{b^N(\tau,k)}\Big)^4\Big)^\frac{n}{2}\Big]^{\frac{2}{n}}\\
	&\lesssim (2N)^{4\alpha}\int_{t_{i-1}}^{t_i} \E \Big[ \Big(\Big(\frac{\xi^N(s,k)}{b^N(s,k)}\Big)^4 Q(k,s,t_{i+1})\, e^{8\gamma_N Q(k,s,t_{i+1})}\Big)^\frac{n}{2}\Big]^{\frac{2}{n}} ds\\
	&\lesssim C (2N)^{4\alpha}\int_{t_{i-1}}^{t_i} g^N_n(s)^{\frac{2}{n}} ds\;,
\end{equs}
uniformly over all $N\geq 1$, all $i\geq 1$ and all $k$. On the other hand, when $i=0$ we have the following bound
\begin{equ}
	\E \Big[ \Big( \sum_{\tau \in \cI_0\cap J_k} \Big(\frac{\xi^N(\tau,k)}{b^N(\tau,k)}\Big)^4\Big)^\frac{n}{2}\Big]^{\frac{2}{n}} \lesssim \Big(\frac{\xi^N(0,k)}{b^N(0,k)}\Big)^4 \E\Big[Q(k,0,t_{1})^{\frac{n}{2}}e^{2n\gamma_N Q(k,0,t_{1})}\Big]\lesssim 1\;,
\end{equ}
uniformly over all $k$ and all $N\geq 1$.\\
Observe that
\begin{equ}
p^N_{t-s}(k,\ell) = e^{-2c_N(t-s)} \sum_{n\geq 0} \frac{(2c_N(t-s))^n}{n!} \mathbf{p}_n(k,\ell)\;,
\end{equ}
where $\mathbf{p}_n(k,\ell)$ is the probability that a discrete-time random walk, killed upon hitting $0$ and $2N$ and started from $k$, reaches $\ell$ after $n$ steps. Therefore, we easily deduce that $\sup_{s\in\cI_i, s<t}q^N_{s,t}(\ell,k) \lesssim q^N_{t_i,t}(\ell,k)$. Using the two bounds of Lemma \ref{Lemma:HeatKernel}, we get
\begin{equs}
	\sum_{k}\sup_{s\in\cI_i,s<t} q^N_{s,t}(k,\ell)^4 &\lesssim \sum_{k} q^N_{t_i,t}(k,\ell)^4 \lesssim \sup_{k} q^N_{t_i,t}(k,\ell)^3 \sum_{k} q^N_{t_i,t}(k,\ell)\\
	&\lesssim b^N(t,\ell)^4 \Big(1\wedge \frac{1}{\sqrt{t-t_i}\,(2N)^{2\alpha}}\Big)\;,
\end{equs}
	uniformly over all $N\geq 1$ and $i\geq 0$. Putting everything together, we obtain
	\begin{equs}
		\E \Big[ \big[ D^t_\cdot(k)\big]_t^{\frac{n}{2}}\Big]^\frac{2}{n} &\lesssim b^N(t,\ell)^4\Big( 1 + \int_0^t \frac{g^N_n(s)^\frac{2}{n}}{\sqrt{t-s}\,(2N)^{2\alpha}} ds \Big)\;,
	\end{equs}
	and the required bound follows by Jensen's inequality, thus concluding the proof.
\end{proof}

The following two lemmas control the moments of the space and time increments of the process: their proofs follow from classical techniques so that we omit the details.

\begin{lemma}\label{Lemma:IncrSpaceKPZ}
Fix $\epsilon > 0$, $\beta\in (0,1/2)$ and a compact set $K \subset [0,T)$. For any $n\geq 1$, we have
\begin{equ}
	\E\Big[ \big|\xi^N(t,\ell')-\xi^N(t,\ell)\big|^{2n} \Big]^{\frac{1}{2n}} \lesssim \Big|\frac{\ell-\ell'}{(2N)^{2\alpha}} \Big|^{\beta}\;,
\end{equ}
uniformly over all $t\in K$, all $\ell,\ell'\in B_\epsilon^N(t)$ and all $N\geq 1$.
\end{lemma}

\begin{lemma}\label{Lemma:IncrTimeKPZ}
Fix $\epsilon > 0$, $\beta\in(0,1/4)$ and a compact set $K \subset [0,T)$. For any $n\geq 1$, we have
\begin{equ}
	\E\Big[ \big|\xi^N(t',\ell)-\xi^N(t,\ell)\big|^{2n} \Big]^{\frac{1}{2n}} \lesssim |t'-t|^\beta + \frac{1}{(2N)^{\alpha}}\;,
\end{equ}
uniformly over all $N\geq 1$, all $t<t'\in K$ and all $\ell\in B_\epsilon^N(t')$.
\end{lemma}

\noindent Then, we get the following result.

\begin{proposition}
Fix $t_0 \in [0,T)$. The sequence $\xi^N$ is tight in $\bbD([0,t_0],\cC(\R))$, and any limit is continuous in time.
\end{proposition}
\begin{proof}
One introduces a piecewise linear time-interpolation $\bar{\xi}^N$ of our process $\xi^N$, namely we set $t_N :=\lfloor t(2N)^{4\alpha}\rfloor$ and
\begin{equ}
	\bar{\xi}^N (t,\cdot) := \big(t_N+1-t(2N)^{4\alpha}\big)\xi^N\Big(\frac{t_N}{(2N)^{4\alpha}},\cdot\Big) + \big(t(2N)^{4\alpha}-t_N\big)\xi^N\Big(\frac{t_N+1}{(2N)^{4\alpha}},\cdot\Big)\;.
\end{equ}
Using Lemmas \ref{Lemma:IncrSpaceKPZ} and \ref{Lemma:IncrTimeKPZ}, it is simple to show that the space-time H\"older semi-norm of $\bar{\xi}^N$ on compact sets of $[0,T)\times\R$ has finite moments of any order, uniformly over all $N\geq 1$. Additionally, the proof of Lemma 4.7 in~\cite{BG97} carries through, and ensures that $\xi^N-\bar{\xi}^N$ converges to $0$ uniformly over compact sets of $[0,T)\times\R$ in probability. All these arguments provide the required control on the space-time increments of $\xi^N$ to ensure its tightness, following the calculation below Proposition 4.9 in~\cite{BG97}.
\end{proof}

We use the notation $\langle f,g \rangle$ to denote the inner product of $f$ and $g$ in $L^2(\R)$. Similarly, for all maps $f,g:[0,2N] \rightarrow \R$, we set
\begin{equ}
\langle f , g \rangle_N := \frac{1}{(2N)^{2\alpha}} \sum_{k=1}^{2N-1} f(k)g(k)\;.
\end{equ}

To conclude the proof of Theorem \ref{Th:KPZ}, it suffices to show that any limit point $\xi$ of a converging subsequence of $\xi^N$ satisfies the following martingale problem (see Proposition 4.11 in~\cite{BG97}).

\begin{definition}[Martingale problem]
Let $(\xi(t,x),t\in [0,T), x\in \R)$ be a continuous process satisfying the following two conditions. Let $t_0\in [0,T)$. First, there exists $a > 0$ such that
\begin{equ}
\sup_{t\leq t_0}\sup_{x\in \R} e^{-a|x|} \E\big[\xi(t,x)^2\big] < \infty\;.
\end{equ}
Second, for all $\varphi\in\cC^\infty_c(\R)$, the processes
\begin{equs}
M(t,\varphi) &:= \langle \xi(t),\varphi\rangle - \langle \xi(0),\varphi\rangle - \frac12 \int_0^t \langle \xi(s),\varphi''\rangle ds\;,\\
L(t,\varphi) &:= M(t,\varphi)^2 - 4 \int_0^t \langle \xi(s)^2,\varphi^2\rangle ds\;,
\end{equs}
are local martingales on $[0,t_0]$. Then, $\xi$ is a solution of (\ref{mSHE}) on $[0,T)$.
\end{definition}
The first condition is a simple consequence of Proposition \ref{Prop:BoundMomentsKPZ}. To prove that the second condition is satisfied, we introduce the discrete analogues of the above processes. For all $\varphi\in\cC^\infty_c(\R)$, the processes
\begin{equs}
M^N(t,\varphi) = \langle \xi^N(t),\varphi\rangle_N - \langle \xi^N(0),\varphi\rangle_N - \frac12 \int_0^t \langle \xi^N(s),(2N)^{4\alpha} \Delta \varphi\rangle_N ds\;,\\
L^N(t,\varphi) = M^N(t,\varphi)^2 - \frac{2\lambda_N}{(2N)^{2\alpha}} \int_0^t \langle \xi^N(s)^2,\varphi^2\rangle_N ds + R^N_1(t,\varphi) + R^N_2(t,\varphi)\;,
\end{equs}
are martingales, where
\begin{equs}
R^N_1(t,\varphi) &:= -\frac{\lambda_N}{(2N)^{2\alpha}}\int_0^t \langle \xi^N(s) \Delta \xi^N(s) , \varphi^2 \rangle_N ds\;,\\
R^N_2(t,\varphi) &:= (2N)^{2\alpha} \int_0^t \langle \nabla^+\xi^N(s) \nabla^- \xi^N(s) , \varphi^2 \rangle_N ds\;.
\end{equs}
If we show that $R^N_1(t,\varphi)$ and $R^N_2(t,\varphi)$ vanish in probability when $N\rightarrow\infty$, then passing to the limit on a converging subsequence, we easily deduce that the martingale problem above is satisfied. Below, we will be working on $[N-A(2N)^{2\alpha},N+A(2N)^{2\alpha}]$ where $A$ is a large enough value such that $[-A,A]$ contains the support of $\varphi$. The moments of $\xi^N$ on this interval are of order $1$ thanks to Proposition \ref{Prop:BoundMomentsKPZ}. Since $|\Delta \xi^N| \lesssim \gamma_N \xi^N$, we have
\begin{equ}
\E\big[ |R^N_1(t,\varphi)| \big] \lesssim \gamma_N \int_0^t \frac{1}{(2N)^{2\alpha}} \sum_k \varphi^2\Big(\frac{k-N}{(2N)^{2\alpha}}\Big) ds \lesssim \gamma_N\;,
\end{equ}
so that $R^N_1(t,\varphi)$ converges to $0$ in probability as $N\rightarrow\infty$. To prove that $R^N_2$ converges to $0$ in probability, it suffices to apply the following delicate estimate which is the analogue of Lemma 4.8 in~\cite{BG97}.

\begin{proposition}\label{Prop:DelicateKPZ}
There exists $\kappa > 0$ such that for all $A>0$, we have
\begin{equ}
\E\Big[\big|\E\big[\nabla^+ \xi^N(t,\ell) \nabla^- \xi^N(t,\ell) \,|\,\cF_s\big]\big|\Big] \lesssim \frac{1}{(2N)^{2\alpha+\kappa}\sqrt{t-s}}\;,
\end{equ}
uniformly over all $\ell\in[N-A(2N)^{2\alpha},N+A(2N)^{2\alpha}]$, all $s<t$ in a compact set of $[N^{-\alpha},T)$ and all $N\geq 1$.
\end{proposition}

\noindent To prove this proposition, we need to collect some preliminary results. Recall the decomposition (\ref{Eq:DiscreteHC}). If we set
\begin{equ}
K^N_{t-r}(k,\ell) := \nabla^+ p^N_{t-r}(k,\ell) \nabla^- p^N_{t-r}(k,\ell)\;,
\end{equ}
(here the gradients act on the variable $\ell$), then using the martingale property of $N^t_\cdot(\ell)$ we obtain for all $s\leq t$
\begin{equs}
\E\Big[\nabla^+ \xi^N(t,\ell) \nabla^- \xi^N(t,\ell)  \, | \, \cF_s \Big] &= \big(\nabla^+ I^{N}(t,\ell) + \nabla^+ N^t_s(\ell)\big)\big(\nabla^- I^{N}(t,\ell) + \nabla^- N^t_s(\ell)\big)\\
&\quad+ \E\bigg[\int_{s}^t \sum_{k=1}^{2N-1} K^N_{t-r}(k,\ell) \,d\langle M^N(\cdot,k)\rangle_r\, \Big| \, \cF_s\bigg]\;.
\end{equs}
Set
\begin{equ}
f^N_s(t,\ell) := \E\bigg[\Big|\E\Big[\nabla^+\xi^N(t,\ell) \nabla^-\xi^N(t,\ell)\,\big|\, \cF_s\Big]\Big|\bigg]\;.
\end{equ}
Fix $\epsilon > 0$. Using the expression of the bracket (\ref{Eq:Bracket}) of $M^N$, we get
\begin{equ}\label{Eq:fN}
f^N_s(t,\ell) \leq D^N_s(t,\ell) + \int_s^t \sum_{k\in B^N_{\epsilon/2}(r)} (2N)^{4\alpha} |K^N_{t-r}(k,\ell)| f^N_s(r,k) dr\;,
\end{equ}
where $D^N_s(t,\ell) = D^{N,1}_s(t,\ell) + D^{N,2}_s(t,\ell) + D^{N,3}_s(t,\ell)$ with
\begin{equs}
D^{N,1}_s(t,\ell) &:= \E\Big[\big|\big(\nabla^+ I^{N}(t,\ell) + \nabla^+ N^t_s(\ell)\big)\big(\nabla^- I^{N}(t,\ell) + \nabla^- N^t_s(\ell)\big) \big|\Big]\;,\\
D^{N,2}_s(t,\ell) &:= \int_s^t \sum_{k\notin B^N_{\epsilon/2}(r)} (2N)^{4\alpha} |K^N_{t-r}(k,\ell)| f^N_s(r,k) dr\;,\\
D^{N,3}_s(t,\ell) &:= \lambda_N\E\bigg[\Big|\E\Big[\int_{s}^t \sum_{k=1}^{2N-1} K^N_{t-r}(k,\ell)\big(\xi^N(r,k)\Delta\xi^N(r,k)+2\xi^N(r,k)^2\big)dr \, \big| \, \cF_s\Big]\Big|\bigg]\;.
\end{equs}
From now on, we fix a compact set $K\subset [0,T)$.
\begin{lemma}\label{Lemma:BoundDN}
Fix $\epsilon > 0$. There exists $\kappa > 0$ such that
\begin{equ}\label{Eq:BoundDNStrong}
D^N_s(t,\ell) \lesssim 1 \wedge \frac{1}{(2N)^{2\alpha+\kappa} \sqrt{t-s}}\;,
\end{equ}
uniformly over all $\ell\in B^N_{\epsilon}(t)$, all $N^{-\alpha} \leq s < t \in K$ and all $N\geq 1$.
\end{lemma}

\begin{proof}
Let us observe that we have the simple bound
\begin{equ}\label{Eq:TrivialBoundfN}
f^N_s(r,k) \lesssim b^N(r,k)^2 \gamma_N^2\;,
\end{equ}
uniformly over all the parameters. Recall also that $b^N(t,\ell)$ is of order $1$ whenever $\ell \in B^N_\epsilon(t)$.\\
Let $\bar{p}^N$ be the discrete heat kernel on the whole line $\Z$ sped up by $2c_N$, see Appendix \ref{Appendix:Kernel}, and set $\bar{K}^N_t(k,\ell) = \nabla^+ \bar{p}^N_t(\ell-k)\nabla^- \bar{p}^N_t(\ell-k)$.\\
\textit{Bound of $D^{N,1}_s$.} It suffices to bound the square of the $L^2$-norms of $\nabla^\pm I^N(t,\ell)$ and $\nabla^\pm N^t_s(\ell)$. By Proposition \ref{Prop:IC}, we deduce that $(\nabla^\pm I^N(t,\ell))^2 \lesssim  N^{-5\alpha}$ uniformly over all $N^{-\alpha} \leq t \in K$, all $\ell \in B^N_\epsilon(t)$ and all $N\geq 1$.\\
We now treat $\nabla^+ N^t_s(\ell)$ (the proof is the same with $\nabla^-$). Using again Lemmas \ref{Lemma:ExpoDecay} and \ref{Lemma:DecaySeriesKernel}, we have
\begin{equs}
\E\Big[\big(\nabla^+ N^t_s(\ell)\big)^2\Big] &\lesssim \E \Big[\sum_{k=1}^{2N-1} \int_0^s\big(\nabla^+ p^N_{t-r}(k,\ell)\big)^2 d\langle M(\cdot,k)\rangle_r\Big] \\
&\lesssim (2N)^{2\alpha} \!\!\sum_{k\in B^N_{\epsilon/2}(r)}\!\! \int_0^s\big(\nabla^+ \bar{p}^N_{t-r}(k,\ell)\big)^2 dr + \cO\big(N^{1+2\alpha}e^{-\delta N^{2\alpha}}\big)\;,
\end{equs}
uniformly over all $\ell\in B^N_\epsilon(t)$, all $t\in K$ and all $N\geq 1$. Using Lemma \ref{Lemma:BoundHeatKernelZ}, we easily deduce that the last expression is bounded by a term of order $1 \wedge 1/(\sqrt{t-s}(2N)^{4\alpha})$ as required.\\
\textit{Bound of $D^{N,2}$.} Using the exponential decay of Lemma \ref{Lemma:ExpoDecay} and (\ref{Eq:TrivialBoundfN}), we deduce that there exists $\delta > 0$ such that
\begin{equ}
\int_s^t \sum_{k\notin B^N_{\epsilon/2}(r)} (2N)^{4\alpha} |K^N_{t-r}(k,\ell)| f^N_s(r,k) dr \lesssim \int_s^t \sum_{k\notin B^N_{\epsilon/2}(r)} (2N)^{2\alpha} e^{-\delta N^{2\alpha}} dr\;,
\end{equ}
uniformly over all $\ell\in B^N_\epsilon(t)$, all $s\leq t \in K$ and all $N\geq 1$. This trivially yields a bound of order $N^{-3\alpha}$ as required.\\
\textit{Bound of $D^{N,3}$.} By Lemmas \ref{Lemma:DecaySeriesKernel} and \ref{Lemma:ExpoDecay}, there exists $\delta>0$ such that $D^{N,3}_s(t,\ell)$ can be rewritten as
\begin{equ}\label{Eq:ExpressionDelicateRewritten}
\lambda_N\E\bigg[\Big|\E\Big[\int_{s}^t \sum_{k\in B^N_{\epsilon/2}(r)} \bar{K}^N_{t-r}(k,\ell)\big(\xi^N(r,k)\Delta\xi^N(r,k)+2\xi^N(r,k)^2\big)dr \, \big| \, \cF_s\Big]\Big|\bigg]\;,
\end{equ}
up to an error of order $N^{2\alpha+1}e^{-\delta N^{2\alpha}}$, uniformly over all $\ell \in B^N_\epsilon(t)$, all $t\in K$ and all $N\geq 1$. The error term satisfies the bound of the statement. We bound separately the two contributions arising in (\ref{Eq:ExpressionDelicateRewritten}). First, using the almost sure bound $|\Delta \xi^N(r,k)| \lesssim \gamma_N \xi^N(r,k)$, we get
\begin{equs}
{}&\lambda_N\E\bigg[\Big|\E\Big[\int_{s}^t \sum_{k\in B^N_{\epsilon/2}(r)} \bar{K}^N_{t-r}(k,\ell) \xi^N(r,k)\Delta\xi^N(r,k)dr \, \big| \, \cF_s\Big]\Big|\bigg]\\
&\lesssim (2N)^{\alpha} \int_s^t \sum_{k\in B^N_{\epsilon/2}(r)} |\bar{K}^N_{t-r}(k,\ell)| dr \;,
\end{equs}
uniformly over all $\ell \in B^N_\epsilon(t)$, all $t\in K$ and all $N\geq 1$. Using Lemma \ref{Lemma:BoundHeatKernelZ}, this easily yields a bound of order $1/(2N)^{2\alpha + \kappa}$ with $\kappa > 0$, as required. Second, we have
\begin{equs}
{}&\int_{s}^t \sum_{k\in B^N_{\epsilon/2}(r)} \bar{K}^N_{t-r}(k,\ell)\E\Big[ \xi^N(r,k)^2 \, \big| \, \cF_s\Big]dr\\
&= \int_{s}^t \sum_{k\in B^N_{\epsilon/2}(r)} \bar{K}^N_{t-r}(k,\ell)\E\Big[ \xi^N(r,k)^2 - \xi^N(t,\ell)^2 \, \big| \, \cF_s\Big]dr\\
&+ \int_{s}^t \sum_{k\in B^N_{\epsilon/2}(r)} \bar{K}^N_{t-r}(k,\ell)dr \, \E\Big[ \xi^N(t,\ell)^2 \, \big| \, \cF_s\Big]\;.
\end{equs}
Using an integration by parts, we get
\begin{equ}
\int_{s}^t \sum_{k\in B^N_{\epsilon/2}(r)} \bar{K}^N_{t-r}(k,\ell)dr = \int_{t-s}^\infty \sum_{k\in\Z} \bar{K}^N_{r}(k,\ell)dr - \int_{s}^t \sum_{k\notin B^N_{\epsilon/2}(r)} \bar{K}^N_{t-r}(k,\ell)dr\;.
\end{equ}
The second term on the right can be bounded using Lemma \ref{Lemma:ExpoDecay}: it has a negligible contribution. Using Lemma \ref{Lemma:BoundHeatKernelZ} on the first term, we easily deduce that
\begin{equ}
\lambda_N \E\bigg[\Big|\int_{s}^t \sum_{k\in B^N_{\epsilon/2}(r)} \bar{K}^N_{t-r}(k,\ell)dr \, \E\Big[ \xi^N(t,\ell)^2 \, \big| \, \cF_s\Big]\Big|\bigg] \lesssim 1 \wedge \frac{1}{\sqrt{t-s} (2N)^{4\alpha}}\;,
\end{equ}
uniformly over all $\ell \in B^N_\epsilon(t)$, all $t\in K$ and all $N\geq 1$. On the other hand, for any given $\beta \in (0,1/4)$, the Cauchy-Schwarz inequality together with Lemmas \ref{Lemma:IncrSpaceKPZ} and \ref{Lemma:IncrTimeKPZ} yields
\begin{equs}
\E\Big[\big|\xi^N(r,k)^2 - \xi^N(t,\ell)^2 \big| \Big] &\lesssim \E\Big[\big(\xi^N(r,k)+\xi^N(r,\ell)\big)^2\Big]^{\frac12}\E\Big[\big(\xi^N(r,k)-\xi^N(r,\ell)\big)^2\Big]^{\frac12}\\
&+ \E\Big[\big(\xi^N(r,\ell)+\xi^N(t,\ell)\big)^2\Big]^{\frac12}\E\Big[\big(\xi^N(r,\ell)-\xi^N(t,\ell)\big)^2\Big]^{\frac12}\\
&\lesssim 1 \wedge \Big( \Big|\frac{\ell-k}{(2N)^{2\alpha}}\Big|^{2\beta} + |t-r|^\beta + \frac{1}{(2N)^{\alpha}} \Big) \;, 
\end{equs}
uniformly over all $\ell \in B^N_\epsilon(t)$, all $k\in B^N_{\epsilon/2}(r)$, all $r\leq t \in K$ and all $N\geq 1$. Using Lemma \ref{Lemma:BoundHeatKernelZ}, it is simple to deduce the existence of $\kappa \in (0,1)$ such that
\begin{equ}
\lambda_N \E\bigg[\Big|\int_{s}^t \sum_{k\in B^N_{\epsilon/2}(r)} \bar{K}^N_{t-r}(k,\ell)\E\Big[ \xi^N(r,k)^2 - \xi^N(t,\ell)^2 \, \big| \, \cF_s\Big]dr \Big|\bigg] \lesssim \frac{1}{(2N)^{2\alpha+\kappa}}\;,
\end{equ}
uniformly over all $s<t \in K$, all $\ell \in B^N_\epsilon(t)$ and all $N\geq 1$.
\end{proof}

\noindent We have all the elements at hand to prove the main result of this section.
\begin{proof}[of Proposition \ref{Prop:DelicateKPZ}]
Iterating (\ref{Eq:fN}) and using Lemma \ref{Lemma:BoundK} and the bound (\ref{Eq:TrivialBoundfN}), we deduce that
\begin{equ}
f^N_s(t,\ell) \leq D^N_s(t,\ell) + \sum_{n\geq 1} H_s(t,\ell,n)\;,
\end{equ}
where for all $n\geq 1$, we set $t_{n+1}=t$, $k_{n+1}=\ell$ and
\begin{equ}
H_s(t,\ell,n):= \int\limits_{s\leq t_1\leq \ldots \leq t_{n} \leq t} \sum_{k_i \in B_{\epsilon/2}^N(t_i)} D_s^N(t_1,k_1) \prod_{i=1}^{n} (2N)^{4\alpha}|K^N_{t_{i+1}-t_i}(k_i,k_{i+1})| dt_i\;.
\end{equ}
By Lemma \ref{Lemma:BoundDN}, we already know that $D^N_s(t,\ell)$ satisfies the bound of the statement of Proposition \ref{Prop:DelicateKPZ}. To conclude the proof of the proposition, we only need to show that this is also the case for the sum over $n\geq 1$ of $H_s(t,\ell,n)$.\\
Fix $A>0$. Let $n_0=c\log N$, for an arbitrary $c > -3\alpha/\log\beta$, where $\beta < 1$ is taken from Lemma \ref{Lemma:BoundK}. Using Lemmas \ref{Lemma:BoundDN} and  \ref{Lemma:BoundK}, we easily deduce that $H_s(t,\ell,n) \lesssim \beta^n$ uniformly over all $n\geq 1$, all $\ell \in \{N-A(2N)^{2\alpha}, N+A(2N)^{2\alpha}\}$ and all $N^{-\alpha} \leq s \leq t \in K$.  Given the definition of $n_0$, we deduce that $\sum_{n\geq n_0}H_s(t,\ell,n) \lesssim (2N)^{-3\alpha}$ uniformly over the same set of parameters, as required.\\
Let us now treat $\sum_{n<n_0} H_s(t,\ell,n)$. We introduce
\begin{equ}
A_s(t,\ell,n):= \int\limits_{s\leq t_1\leq \ldots \leq t_{n} \leq t} \sum_{k_i \in B^N_\epsilon(t_i)} D_s^N(t_1,k_1) \prod_{i=1}^{n} (2N)^{4\alpha}|K^N_{t_{i+1}-t_i}(k_i,k_{i+1})| dt_i\;.
\end{equ}
By Lemma \ref{Lemma:BoundDN}, we have
\begin{equ}
A_s(t,\ell,n)\lesssim \!\!\!\! \int\limits_{s\leq t_1\leq \ldots \leq t_{n} \leq t}\!\! \sum_{k_i \in B^N_\epsilon(t_i)} \frac{1}{(2N)^{2\alpha+\kappa}\sqrt{t_1-s}} \prod_{i=1}^{n} (2N)^{4\alpha}|K^N_{t_{i+1}-t_i}(k_i,k_{i+1})| dt_i\;.
\end{equ}
If we restrict the domain of integration to those $t_1$ such that $t_1-s \geq (t-s)/(n+1)$, then a simple calculation based on Lemma \ref{Lemma:BoundK} ensures that this restricted integral is bounded by a term of order
\begin{equ}
\frac{\sqrt{n+1} \beta^n}{(2N)^{2\alpha+\kappa}\sqrt{t-s}} \lesssim \frac{1}{(2N)^{2\alpha+\kappa}\sqrt{t-s}}\;,
\end{equ}
for all $n\leq n_0$. On the other hand, when $t_1-s < (t-s)/(n+1)$ there is at least one increment $t_{i+1}-t_i$ which is larger than $(t-s)/(n+1)$. By symmetry, let us consider the case $i=1$. By Lemma \ref{Lemma:DecaySeriesKernel}, we can replace $K^N_{t_{2}-t_1}(k_1,k_{2})$ with $\bar{K}^N_{t_{2}-t_1}(k_1,k_{2})$ up to a negligible term. By Lemma \ref{Lemma:BoundHeatKernelZ}, we bound the sum over $k_1$ of $|\bar{K}^N_{t_{2}-t_1}(k_1,k_{2})|$ by a term of order $(2N)^{-4\alpha}(t_2-t_1)^{-1}$ thus yielding the bound
\begin{equs}
{}&\int_s^{s+\frac{t-s}{n+1}} \sum_{k_1 \in B^N_\epsilon(t_1)} \frac{1}{(2N)^{2\alpha+\kappa}\sqrt{t_1-s}}\, (2N)^{4\alpha}|K^N_{t_{2}-t_1}(k_1,k_{2})| dt_1\\
&\lesssim \int_{s}^{s+\frac{t-s}{n+1}} \frac{1}{(2N)^{2\alpha+\kappa}\sqrt{t_1-s}\,(t_2-t_1)}\, dt_1 + \cO(N^{1+4\alpha}e^{-\delta N^{2\alpha}})\\
&\lesssim \frac{\sqrt{\frac{t-s}{n+1}}}{(2N)^{2\alpha+\kappa}\frac{t-s}{n+1}} \lesssim \frac{1}{(2N)^{2\alpha+\kappa'} \sqrt{t-s}}\;,
\end{equs}
for all $\kappa'\in (0,\kappa)$ and all $n< n_0$. Using Lemma \ref{Lemma:BoundK}, we can bound the integral over $t_2,\ldots,t_n$ of the remaining terms by a term of order $\beta^{n-1}$. Consequently, we have proved that there exists $\kappa'>0$ such that
\begin{equ}\label{Eq:BoundAs}
A_s(t,\ell,n) \lesssim \frac{\beta^{n-1}}{(2N)^{2\alpha + \kappa'}\sqrt{t-s}}\;,
\end{equ}
uniformly over all $\ell \in [N-A(2N)^{2\alpha},N+A(2N)^{2\alpha}]$, all $t\in K$, all $s\in [0,t]$, all $n < n_0$ and all $N\geq 1$.\\
Finally, we set $B_s(t,\ell,n) := H_s(t,\ell,n) - A_s(t,\ell,n)$. We can replace each occurrence of $p^N$ by $\bar{p}^N$ up to a negligible term, using Lemma \ref{Lemma:DecaySeriesKernel}. Among the parameters $k_1,\ldots,k_n$ involved in the definition of $B_s(t,\ell,n)$, at least one them, say $k_{i_0}$, belongs to $B^N_{\epsilon/2}(t_{i_0})\backslash B^N_\epsilon(t_{i_0})$. Then, using the bound $\big| \bar{K}^N_{t}(k,\ell) \big| \leq \bar{p}^N_{t}(k,\ell)$ together with the semigroup property of the discrete heat kernel at the second line and the exponential decay of Lemma \ref{Lemma:ExpoDecay}, we get
\begin{equs}
{}&\sum_{k_{i_0 +1},\ldots,k_n} \prod_{i=i_0}^n \big| \bar{K}^N_{t_{i+1}-t_i}(k_i,k_{i+1}) \big|\\
&\leq \sum_{k_{i_0 +1},\ldots,k_n} \prod_{i=i_0}^n \bar{p}^N_{t_{i+1}-t_i}(k_i,k_{i+1})= \bar{p}^N_{t-t_{i_0}}(k_{i_0},\ell)\lesssim e^{-\delta N^{2\alpha}}\;,
\end{equs}
uniformly over all the parameters. Using Lemma \ref{Lemma:BoundDN}, one easily gets
\begin{equ}
B_s(t,\ell,n) \lesssim (2N)^{n4\alpha} e^{-\delta(2N)^{2\alpha}}\;,
\end{equ}
uniformly over all $n\geq 1$, all $s<t\in K$ and all $\ell \in [N-A(2N)^{2\alpha},N+A(2N)^{2\alpha}]$. Given the definition of $n_0$, we deduce that the sum over all $n<n_0$ of the latter is negligible w.r.t.~$(2N)^{-3\alpha}$, uniformly over the same set of parameters. This concludes the proof.
\end{proof}

\section{Appendix}

\subsection{Martingale inequalities}\label{Appendix:Mgale}

Let $X(t),t\geq 0$ be a c\`adl\`ag, mean zero, square-integrable martingale. Let $\langle X \rangle_t, t\geq 0$ denote the bracket of $X$, that is, the unique predictable process such that $X^2-\langle X \rangle$ is a martingale. Let $[X]_t$ denote its quadratic variation: in the case where the martingale is of finite variation, we have ${[X]}_t = \sum_{\tau \in (0,t]} (X_\tau - X_{\tau-})^2$. It happens that the process $D_t = [X]_t - \langle X \rangle_t$ is also a martingale. We will rely on the following Burkholder-Davis-Gundy inequalities. For every $p \geq 2$, there exists $c(p) > 0$ such that
\begin{equs}\label{Eq:BDG2}
	\E\big[ |X_t|^p \big]^{\frac{1}{p}} &\leq c(p)\Big( \E\Big[ \langle X \rangle_t^{\frac{p}{2}} \Big]^{\frac{1}{p}} + \E\Big[ [ D ]_t^{\frac{p}{4}} \Big]^{\frac{1}{p}}\Big)\;,\\
	\E\Big[ \sup_{s\leq t} |X_s|^p \Big]^{\frac{1}{p}} &\leq c(p) \Big( \E\Big[ \langle X \rangle_t^{\frac{p}{2}} \Big]^{\frac{1}{p}} + \E\Big[ \sup_{s\leq t}|X_s-X_{s-}|^p \Big]^{\frac{1}{p}} \Big)\;.\label{Eq:BDG3}
\end{equs}

\subsection{Discrete heat kernel estimates}\label{Appendix:Kernel}

We introduce the fundamental solution $p^N_t(k,\ell)$ of the discrete heat equation
\begin{equs}\label{DiscreteHeat}
	\begin{cases}
	\partial_t p^N_t(k,\ell) = c_N \Delta p^N_t(k,\ell)\;,\\
	p^N_0(k,\ell) = \delta_k(\ell)\;,\\
	p^N_t(k,0) = p^N_t(k,2N) = 0\;,\end{cases}
\end{equs}
for all $k,\ell \in \{1,\ldots,2N-1\}$, as well as its analogue $\bar{p}^N_t(\ell)$ on $\Z$:
\begin{equs}\label{DiscreteHeatKPZWholeLine}
	\begin{cases}\partial_t \bar{p}^N_t(\ell) = c_N \Delta \bar{p}^N_t(\ell)\;,\\
	\bar{p}^N_0(\ell) = \delta_0(\ell)\;,\end{cases}
\end{equs}
for all $\ell \in \Z$. The latter is more tractable than the former since it is translation invariant. Using a coupling between a simple random walk on $\Z$ and a simple random walk killed at $0$ and $2N$, we get the elementary bound $p^N_t(k,\ell) \leq \bar{p}^N_t(\ell-k)$ for all $k,\ell \in \{1,\ldots,2N-1\}$ and all $t\geq 0$. The following estimates are classical, so we omit the proofs.

\begin{lemma}\label{Lemma:BoundHeatKernelZ}
We have $\sum_k \bar{p}^N_t(k)|k| \lesssim \sqrt{c_N t}\;,\quad$ $|\bar{p}^N_t(\ell)| \lesssim (c_N t)^{-\frac12}$ as well as
\begin{equ}
\sum_{k\in\Z} |\nabla \bar{p}^N_t(k)| \leq 2\bar{p}^N_t(0)\;,\quad \sum_{k\in\Z} |\nabla \bar{p}^N_t(k)| |k| \lesssim 1\;,\quad |\nabla \bar{p}^N_t(\ell)| \lesssim 1\wedge \frac{1}{tc_N}\;,
\end{equ}
uniformly over all $\ell\in\Z$, all $t\geq 0$ and all $N\geq 1$.
\end{lemma}

From now on, we take $2c_N=(2N)^{4\alpha}$. Recall the definition of $q^N$ from (\ref{Def:qN}).
\begin{lemma}\label{Lemma:HeatKernel}
Uniformly over all $0\leq s < t$, all $\ell \in \{1,\ldots,2N-1\}$ and all $N\geq 1$, we have
\begin{equs}\label{Eq:BoundHeatMass}
	\sum_{k=1}^{2N-1} q^N_{s,t}(k,\ell) \lesssim b^N(t,\ell)\;,\quad \sup_k q^N_{s,t}(k,\ell) \lesssim b^N(t,\ell)\Big( 1\wedge \frac{1}{\sqrt{t-s}\, (2N)^{2\alpha}} \Big)\;.\;\;
\end{equs}
\end{lemma}
\begin{proof}
By symmetry, it is sufficient to prove the lemma under the further assumption $\ell \leq N$. If $b^N(s,k) \leq 3$, then we have $b^N(s,k)/b^N(t,\ell) \leq 3$ while if $b^N(s,k) > 3$, then a simple calculation ensures that
\begin{equ}
\frac{b^N(s,k)}{b^N(t,\ell)} \lesssim \begin{cases} e^{-\lambda_N(t-s) + \gamma_N(\ell-k)}\quad\mbox{ if }k\leq N\;,\\
 e^{-\lambda_N(t-s) + \gamma_N(k-\ell)}\quad\mbox{ if }k\geq N\;,\end{cases}
\end{equ}
uniformly over all $s < t$, all $k\in \{1,\ldots,2N-1\}$, all $\ell\in\{1,\ldots,N\}$ and all $N\geq 1$. Therefore, it suffices to show that
\begin{equ}\label{Eq:BoundKernelCrude}
\sum_{k=1}^{2N-1} p^N_{t-s}(k,\ell) z^N_{s,t}(k,\ell) \lesssim 1\;,\quad p^N_{t-s}(k,\ell)z^N_{s,t}(k,\ell) \lesssim 1\wedge \frac{1}{\sqrt{t-s}(2N)^{2\alpha}}\;, 
\end{equ}
with $z^N_{s,t}(k,\ell)$ being either $1$ or $\exp(-\lambda_N(t-s) + \gamma_N|\ell-k|)$. When $z^N_{s,t}(k,\ell)=1$, the bounds are classical, see for instance Lemma 26 of~\cite{EthLab15}. We turn to the other case. Since $k\mapsto e^{a k}$ is an eigenvector of the discrete Laplacian on $\Z$ with eigenvalue $2(\cosh a -1)$, we deduce that
\begin{equ}
\sum_{k=1}^{2N-1} p^N_{t-s}(k,\ell) e^{\gamma_N |\ell-k|} \leq \sum_{k\in\Z} \bar{p}^N_{t-s}(\ell-k) e^{\gamma_N |\ell-k|}\leq 2 e^{2c_N(t-s)( \cosh \gamma_N - 1)} = 2 e^{\lambda_N (t-s)}\;,
\end{equ}
thus yielding the first bound. To get the second bound, it suffices to use the Fourier decomposition of $\bar{p}^N_{t-s}(\cdot)e^{\gamma_N \cdot}$.
\end{proof}

For all $a>0$, $t>0$ and $N\geq 1$, we have
\begin{equ}\label{Eq:LargeDevKernel}
\sum_{k\geq a} \bar{p}^N_t(k) \leq e^{2t c_N g\big(\frac{a}{2c_N t}\big)}\;,
\end{equ}
where $g(x) =\cosh(\argsh x) - x\,\argsh x - 1$ for all $x\in\R$. By studying the function $g(x)/x$, one easily deduces that the term on the r.h.s.~is increasing with $t$.\\
We let $\bar{q}^N_{s,t}(k,\ell) := \bar{p}^N_{t-s}(k,\ell) b^N(s,k)$.

\begin{lemma}\label{Lemma:ExpoDecay}
Fix a compact set $K\subset [0,T)$ and $\epsilon>0$. There exists $\delta >0$ such that
\begin{equ}
q^N_{s,t}(k,\ell) \leq \bar{q}^N_{s,t}(k,\ell) \lesssim e^{-\delta N^{2\alpha}}\;,
\end{equ}
uniformly over all $k\notin B^N_{\epsilon/2}(s)$, all $\ell\in B^N_\epsilon(t)$ and all $0 \leq s \leq t \in K$.
\end{lemma}
\begin{proof}
Let us consider the case where $b^N(s,k)\geq 3$; by symmetry we can assume that $k \in \{1,\ldots,N\}$. Then, we apply (\ref{Eq:LargeDevKernel}) to get
\begin{equ}\label{Eq:IntermediateExpoDecay}
\bar{p}^N_{t-s}(\ell-k) e^{\lambda_N s - \gamma_N k} \leq e^{2(t-s) c_N g\big(\frac{\ell-k}{2c_N (t-s)}\big) + \lambda_N s - \gamma_N k}\;.
\end{equ}
We argue differently according to the value of $\alpha$. If $4\alpha \leq 1$, then $(\ell-k)/c_N$ is bounded away from $0$ uniformly over all $N\geq 1$, all $k\notin B_{\epsilon/2}(s)$ and all $\ell\in B_\epsilon(t)$. Using the concavity of $g$, we deduce that there exists $d>0$ such that the logarithm of the r.h.s.~of (\ref{Eq:IntermediateExpoDecay}) is bounded by
\begin{equ}
-d(\ell-k) + \lambda_N s -\gamma_N k \lesssim -d \epsilon \frac{N}{2}\;,
\end{equ}
thus concluding the proof in that case.\\
We now treat the case $4\alpha > 1$. Let $\eta > 0$. First, we assume that $s \in [0,t-\eta]$. For any $c >1/4!$, we have $g(x)\leq -x^2/2 + cx^4$ for all $x$ in a neighbourhood of the origin. Then, for $N$ large enough we bound the logarithm of the r.h.s.~of (\ref{Eq:IntermediateExpoDecay}) by
\begin{equ}
f(s) = -\frac{1}{2}\frac{(\ell-k)^2}{2c_N(t-s)} + c \frac{(\ell-k)^4}{(2c_N(t-s))^3} + \lambda_N s - \gamma_N k\;.
\end{equ}
A tedious but simple calculation shows the following. There exists $\delta' > 0$, only depending on $\epsilon$, such that $\sup_{s\in [0,t-\eta]} f(s) \leq -\delta' N^{2\alpha}$ for all $N$ large enough. This ensures the bound of the statement in the case where $s\in[0,t-\eta]$.\\
Using the monotonicity in $t$ of (\ref{Eq:LargeDevKernel}), we easily deduce that for all $s\in[t-\eta,t]$, we have
\begin{equ}
\bar{p}^N_{t-s}(\ell-k) e^{\lambda_N s - \gamma_N k}\leq e^{f(t-\eta) + \lambda_N \eta}\;.
\end{equ}
Recall that $\lambda_N$ is of order $N^{2\alpha}$. Choosing $\eta < \delta'$ small enough and applying the bound obtained above, we deduce that the statement of the lemma holds true.\\
The case where $b^N(s,k)$ is smaller than $3$ is simpler, one can adapt the above arguments to get the required bound.
\end{proof}

\begin{proof}[of (\ref{Eq:BoundTailDiscreteKernel})]
The quantity $1-\sum_{k=1}^{2N-1} p^N_{t-s}(k,\ell)$ is equal to the probability that a simple random walk, sped up by $2c_N$ and started from $\ell$, has hit $0$ or $2N$ by time $t-s$. By the reflexion principle, this is smaller than twice
\begin{equ}
\sum_{k\geq \ell} \bar{p}^N_{t-s}(k) + \sum_{k\geq 2N-\ell} \bar{p}^N_{t-s}(k)\;.
\end{equ}
We restrict ourselves to bounding the first term, since one can proceed similarly for the second term. Using (\ref{Eq:LargeDevKernel}), we deduce that it suffices to bound $\exp(2(t-s)c_N g\big(\ell/(2(t-s)c_N)\big) +\lambda_N s)$. This is equal to the l.h.s.~of (\ref{Eq:IntermediateExpoDecay}) when $k=0$, so that the required bound follows from the arguments presented in the last proof.
\end{proof}
Finally, we use the following representation of $p^N$:
\begin{equ}
p^N_t(k,\ell) = \sum_{j\in\Z} \bar{p}^N_t(k+ j4N-\ell) - \bar{p}^N_t(-k+ j4N-\ell)\;.
\end{equ}
The next lemma shows that $p^N_t(k,\ell)$ can be replaced by $\bar{p}^N_t(\ell-k)$ up to some negligible term, whenever $\ell$ is in the $\epsilon$-bulk at time $t$. 

\begin{lemma}\label{Lemma:DecaySeriesKernel}
Fix $\epsilon > 0$ and a compact set $K\subset [0,T)$. There exists $\delta > 0$ such that uniformly over all $s\leq t \in K$, all $k\in\{1,\ldots,2N-1\}$, all $\ell\in B_\epsilon^N(t)$ and all $N\geq 1$, we have
\begin{equ}
\big| p^N_{t-s}(k,\ell) - \bar{p}^N_{t-s}(k,\ell) \big| b^N(s,k) \lesssim e^{-\delta N^{2\alpha}}\;.
\end{equ}
\end{lemma}
\begin{proof}
We only consider the case where $b^N(s,k) > 3$ since the other case is simpler. Observe that there exists $C>0$ such that $\log b^N(s,k) \leq C N^{2\alpha}$ for all $s\in K$ and all $k\in\{1,\ldots,2N-1\}$. Arguing differently according to the relative values of $4\alpha$ and $1$, and using the bound (\ref{Eq:LargeDevKernel}), we deduce that there exists $j_0\geq 1$ such that
\begin{equ}\label{Eq:PeriodicKernel}
\sum_{j\in\Z: |j| \geq j_0} \bar{p}^N_{t-s}(k+ j4N-\ell)b^N(s,k) + \bar{p}^N_{t-s}(-k+ j4N-\ell)b^N(s,k) \lesssim e^{-\delta N^{2\alpha}}\;,
\end{equ}
uniformly over all $s\leq t \in K$, all $k\in\{1,\ldots,2N-1\}$ and all $\ell\in B_\epsilon(t)$. On the other hand, the arguments in the proof of Lemma \ref{Lemma:ExpoDecay} yield that
\begin{equs}
\sum_{j\in\Z: |j| < j_0} \bar{p}^N_{t-s}(-k+ j4N-\ell)b^N(s,k) &\lesssim e^{-\delta N^{2\alpha}}\;,\\
\sum_{j\in\Z: 0 < |j| < j_0,} \bar{p}^N_{t-s}(k+ j4N-\ell)b^N(s,k) &\lesssim e^{-\delta N^{2\alpha}}\;,
\end{equs}
uniformly over the same set of parameters, thus concluding the proof.
\end{proof}

\begin{lemma}\label{Lemma:BoundK}
Fix $\epsilon > 0$. There exist $\beta \in (0,1)$ such that
\begin{equ}
\int_s^t \sum_{k\in B^N_{\epsilon/2}(r)} |K_{t-r}^N(k,\ell)|(2N)^{4\alpha} dr < \beta\;,
\end{equ}
uniformly over all $s \leq t \in K$, all $\ell\in B^N_\epsilon(t)$ and all $N$ large enough.
\end{lemma}
\begin{proof}
Lemma \ref{Lemma:DecaySeriesKernel} ensures that
\begin{equ}
\int_s^t \!\!\sum_{k\in B_{\epsilon/2}(r)}\!\! |K^N_{t-r}(k,\ell)| (2N)^{4\alpha} dr = \int_s^t \!\!\sum_{k\in B_{\epsilon/2}(r)}\!\! |\bar{K}^N_{t-r}(k,\ell)| (2N)^{4\alpha} dr+ \cO(N^{1+4\alpha}e^{-\delta N^{2\alpha}})\;,
\end{equ}
uniformly over all $\ell \in B_\epsilon(t)$, all $t\in K$ and all $N\geq 1$. Lemma A.3 in~\cite{BG97} ensures that the first term on the r.h.s.~is smaller than some $\beta' \in (0,1)$. Since the second term vanishes as $N\rightarrow\infty$, the bound of the statement follows.
\end{proof}

\bibliographystyle{Martin}
\bibliography{library}

\end{document}